\documentclass[12pt]{amsart}

\usepackage[T1]{fontenc}
\usepackage{times}

\usepackage{amsfonts, amsmath, amsthm, amssymb, mathtools}
\usepackage{mathrsfs}
\usepackage{tikz-cd}
\usepackage{hyperref}

\newtheorem{theorem}{Theorem}[section]
\newtheorem{lemma}[theorem]{Lemma}
\newtheorem{corollary}[theorem]{Corollary}
\theoremstyle{definition}
\newtheorem{definition}[theorem]{Definition}
\theoremstyle{remark}
\newtheorem{remark}[theorem]{Remark}
\newtheorem{example}[theorem]{Example}

\DeclareMathOperator{\Spec}{Spec}
\DeclareMathOperator{\ch}{char}
\DeclareMathOperator{\Pic}{Pic}
\DeclareMathOperator{\Alb}{Alb}

\newcommand{\pp}{{\mathbb{P}}}
\newcommand{\Q}{{\mathbb{Q}}}
\newcommand{\Z}{{\mathbb{Z}}}

\newcommand{\R}{{\mathbb{R}}}

\newcommand{\oo}{{\mathcal{O}}}
\newcommand{\h}{\widehat}

\newcommand{\ol}{\overline}
\newcommand{\m}{\mathfrak{m}}

\newcommand{\A}{\mathbb{A}}
\newcommand{\p}{\mathfrak{p}}
\newcommand{\ka}{\kappa}

\newcommand{\re}{\mathrm{red}}
\newcommand{\sing}{\mathrm{sing}}
\newcommand{\reg}{\mathrm{reg}}
\newcommand{\sm}{\mathrm{sm}}
\newcommand{\sh}{\mathrm{sh}}

\newcommand{\sep}{\mathrm{sep}}

\title{semi-abelian reduction of Albanese varieties}
\author{Tai-Hsuan Chung}


\address{Department of Mathematics, National Cheng Kung University, No. 1, Dasyue Rd., Tainan City 70101, Taiwan}
\email{taihsuanchung@gmail.com}

\begin{document}

\begin{abstract}
Let $X_K$ be a projective geometrically normal variety over the fraction field of a DVR. We show that if $X_K$ admits a projective model whose special fibre has at worst nodes in codimension one, then the Albanese variety $\Alb_{X_K/K}$ has semi-abelian reduction over the same base. As an application, we extend the classical nonexistence theorem of Fontaine and Abrashkin for abelian schemes over $\Z$ to the singular setting. For instance, if a flat projective $\Z$-scheme $X$ has normal and geometrically connected fibres, then $H^1(X_{\Q},\oo_{X_{\Q}})=0$. 
\end{abstract}

\maketitle

\tableofcontents

\section{Introduction}

In their seminal paper, Deligne and Mumford proved that a smooth proper curve of genus at least 2 has stable reduction over a discrete valuation ring if and only if its Jacobian has semi-abelian reduction over the same ring \cite[Thm.~2.4]{DM69}. In higher dimension, this classical result raises a natural question: given a variety $X_K$ over the fraction field $K$ of a discrete valuation ring, what geometric conditions on a model of $X_K$ force semi-abelian or good reduction of its Albanese variety $\Alb_{X_K/K}$? In this note, we prove that if a projective geometrically normal variety $X_K$ admits a projective model whose special fibre has at worst nodes in codimension one (resp.~is smooth in codimension one), then $\Alb_{X_K/K}$ has semi-abelian (resp.~good) reduction. In particular, these reduction types are governed entirely by the codimension-one geometry of the special fibre; singularities in codimension at least two play no role (Theorem~\ref{Main_result}).

\subsection{Main result} Throughout this note, $S$ denotes a Dedekind scheme (i.e. a Noetherian, normal, connected scheme of dimension one) with fraction field $K$, and $s\in S$ a closed point. Let $X_K$ be a projective, geometrically normal, and geometrically connected scheme over $K$. A \emph{model} of $X_K$ over the DVR $\oo_{S,s}$ is a flat projective scheme $X\to\Spec\oo_{S,s}$ whose generic fibre is isomorphic to $X_K$. We say that an equidimensional scheme locally of finite type over a field is \emph{geometrically $R_1$} if every codimension-one point is geometrically regular; it is \emph{geometrically $N_1$} if every codimension-one point is either a geometrically regular point or a node (cf. Definition~\ref{definition_node}). We say that $X_K$ has \emph{geometrically $N_1$ (resp.~geometrically $R_1$) reduction} at $s\in S$ if there exists a model $X$ over $\oo_{S,s}$ whose special fibre $X_s$ is geometrically $N_1$ (resp.~geometrically $R_1$). 

Recall that an abelian variety $A_K$ over $K$ has \emph{semi-abelian} (resp.~\emph{good}) reduction at $s\in S$ if the identity component $\mathscr{N}^0_s$ of the special fibre of its N\'eron model $\mathscr{N}$ is a semi-abelian (resp.~abelian) variety (cf. \cite[\S~7.4]{BLR90}). By definition, a \emph{semi-abelian variety} is an extension of an abelian variety by an algebraic torus. 

We now state our main result.

\begin{theorem}\label{Main_result}
If $X_K$ has geometrically $N_1$ (resp.~geometrically $R_1$) reduction at $s\in S$, then its Albanese variety $\Alb_{X_K/K}$ and the Picard variety $(\Pic^0_{X_K/K})_{\re}$ have semi-abelian (resp.~good) reduction at $s\in S$.
\end{theorem}

The proof uses a Bertini-type argument to reduce the problem to curves. The main technical challenge of this note involves formalizing a suitable notion of nodes of codimension one and establishing that it is preserved under taking appropriate hypersurface sections. To this end, we prove a new Bertini-type theorem for non-isolated nodal singularities over an arbitrary field (Theorems~\ref{N1_Bertini_infinite_field} and \ref{N1_Bertini_finite_field}), which may be of independent interest. We defer the detailed discussion of this strategy and the motivation for the $N_1$ condition to Subsection~\ref{Strategy}.

\begin{remark}
Even when $\Alb_{X_K/K}$ itself is trivial, Theorem~\ref{Main_result} provides non-trivial obstructions. By applying the result iteratively to hypersurface sections, one obtains infinitely many successive hypersurface sections of $X_K$ whose Albanese varieties have semi-abelian (resp.~good) reduction (see Corollary~\ref{complete_intersections}). This provides a geometric obstruction to the existence of a geometrically $N_1$ reduction of $X_K$.
\end{remark}

\begin{remark}[Sharpness of the $N_1$ condition] For curves of genus $\ge2$, by \cite[Thm.~2.4]{DM69}, the semi-abelian reduction of the Jacobian forces the special fibre of the minimal regular model of the curve to be geometrically $N_1$. Thus in the curve case, the geometrically $N_1$ condition is optimal. We note, however, that for good reduction, the geometrically $R_1$ condition can be weakened even in the curve case (cf. \cite[Ex.~9.2/8]{BLR90}).
\end{remark}

We discuss two reduction types which are geometrically $N_1$. Let $X_K$ be as in Theorem~\ref{Main_result}, and assume in addition that it is log canonical. We say that $X_K$ has \emph{locally stable reduction at $s\in S$} if it admits a projective model $X$ over $\oo_{S,s}$ such that the special fibre $X_s$ is reduced and the pair $(X,X_s)$ is log canonical. It turns out that $X_s$ is geometrically $N_1$ provided $\ka(s)$ is perfect of odd characteristic.

\begin{corollary}\label{locally_stable_implies_geo_N1}
Assume the residue field $\ka(s)$ is perfect with $\ch \ka(s)\ne2$. If $X_K$ has locally stable reduction at $s\in S$, then $\Alb_{X_K/K}$ and $(\Pic^0_{X_K/K})_{\re}$ have semi-abelian reduction at $s\in S$.
\end{corollary}

Recall that a smooth, projective, and geometrically connected scheme $X_K$ is said to have \emph{semi-stable reduction at $s\in S$} if there exists a regular projective model $X$ over $\oo_{S,s}$ such that each point $x\in X_s$ admits an \'etale neighborhood $U\to X$ of $x$ that is \'etale over 
\[
\Spec\oo_{S,s}[x_1,x_2,\dots, x_{n+1}]/(x_1x_2\cdots x_r-\pi),
\]
where $\pi$ is a uniformizer of $\oo_{S,s}$, $1\le r=r(x)\le n+1$, and $n=\dim X_K$. A standard SNC scheme $\Spec k[x_1,x_2,\dots, x_m]/(x_1x_2\cdots x_r)$ is geometrically $N_1$ (Lemma~\ref{snc_implies_geo_N1}), and thus we obtain

\begin{corollary}\label{semi-stable_implies_semi-abelian}
If $X_K$ has semi-stable reduction at $s\in S$, then $\Alb_{X_K/K}$ and $(\Pic^0_{X_K/K})_{\re}$ have semi-abelian reduction at $s\in S$.
\end{corollary}

The proofs of Corollaries~\ref{locally_stable_implies_geo_N1} and \ref{semi-stable_implies_semi-abelian} are deferred to Section~\ref{Proofs_corollaries}.

\begin{remark}\label{Grothendieck_monodromy}
Corollary~\ref{semi-stable_implies_semi-abelian} holds even if $X_K$ is only assumed to be proper instead of projective. This follows from Grothendieck's theory of nearby and vanishing cycles and his criterion for semi-abelian reduction of abelian varieties; see Subsection~\ref{Comparison} for a detailed discussion.
\end{remark}

Applying Theorem~\ref{Main_result} with $X_K$ an abelian variety gives a geometric perspective of the N\'eron--Ogg--Shafarevich criterion (cf. \cite[Thm.~1]{ST68}).

\begin{corollary}\label{good_iff_geo_R1}
An abelian variety over $K$ has good reduction at $s\in S$ if and only if it has geometrically $R_1$ reduction at $s\in S$.
\end{corollary}

For the $N_1$ analogue of Corollary~\ref{good_iff_geo_R1}, J\'anos Koll\'ar pointed out that the equal characteristic $0$ case follows from \cite[Cor.~5, (5.1)]{Kol26}. The analogue also extends to perfect residue fields of odd characteristic provided
\cite[Thm.~42]{Kol26} holds in that setting.

\begin{corollary}\label{SAR_iff_geo_N1}
Suppose $S$ is of equal characteristic $0$. Then an abelian variety over $K$ has semi-abelian reduction at $s\in S$ if and only if it has geometrically $N_1$ reduction at $s\in S$.
\end{corollary}

\subsection{Arithmetic nonexistence} A classical problem in \mbox{arithmetic geometry} asks which smooth proper morphisms $X\to\Spec\Z$ exist, namely, which smooth proper varieties over $\Q$ have everywhere good reduction. A result of Ogg \cite{Ogg66}, which he attributes to Tate, shows that no elliptic curve exists over $\Z$. This was generalized independently by Fontaine \cite{Fon85} and Abrashkin \cite{Abr88}, who established that there is no nonzero abelian variety over $\Q$ with everywhere good reduction. Afterwards, Abrashkin \cite{Abr90} and Fontaine \cite{Fon93} proved that a smooth proper variety $X_{\Q}$ with everywhere good reduction satisfies strong vanishing results for low Hodge numbers, i.e. $H^i(X_{\Q},\Omega^j_{X_{\Q}})=0$ for all $i+j\le3$ with $i\ne j$. Recently, this classical program has expanded to broader geometric settings, yielding new progress on smooth models over $\Z$ in several directions: the nonexistence result for Enriques surfaces by Schr\"oer \cite{Sch23}, the complete existence and nonexistence picture for Mukai varieties of genus $7$ in dimensions at most 10 by Ito et al.~\cite{IKTT25}, and the birational classification of smooth projective models of surfaces with negative Kodaira dimension by Bernasconi et al.~\cite{BMP26}.

These classical and recent nonexistence results are formulated for smooth proper models over $\Z$, and therefore they naturally lead to the following question: to what extent can such arithmetic obstructions be circumvented by allowing singular reductions? In other words, can a variety ``gain'' the ability to exist over $\Z$ by permitting sufficiently complicated singularities in its closed fibres?

Theorem~\ref{Main_result} provides a largely negative answer for irregular normal projective varieties. By establishing a direct reduction to the Fontaine--Abrashkin theorem via Albanese varieties, we generalize their classical nonexistence result to the singular setting:

\begin{theorem}\label{everywhere_R1_reduction}
If $X_{\Q}$ is a normal, projective, and geometrically connected scheme over $\Q$ with everywhere $R_1$ reduction, then $\Alb_{X_{\Q}/\Q}$ is trivial.\break That is, $H^1(X_{\Q},\oo_{X_\Q})=0$.
\end{theorem}

In contrast to the classical smooth setting, the reductions appearing in Theorem~\ref{everywhere_R1_reduction} are permitted to possess arbitrarily complicated singularities away from codimension one. We emphasize that the closed fibres, being merely $R_1$, may all be nonreduced. We are not aware of previous arithmetic nonexistence criteria over $\Z$ that remain applicable when nonreduced closed fibres are allowed.  

\begin{proof}[Proof of Theorem~\ref{everywhere_R1_reduction}]
Since each closed fibre is defined over a perfect field, being $R_1$ is the same as being geometrically $R_1$. Therefore Theorem~\ref{Main_result} implies that $\Alb_{X_{\Q}/\Q}$ has everywhere good reduction, and hence is trivial by Fontaine--Abrashkin.
\end{proof}

\begin{remark}
Theorem~\ref{everywhere_R1_reduction} also holds if $\Q$ is replaced by $\Q(\sqrt{-1})$, $\Q(\sqrt{-3})$, or $\Q(\sqrt{5})$, because \cite[Cor., p.~517]{Fon85} guarantees that no nonzero abelian varieties over these fields have everywhere good reduction.
\end{remark}

We note that Theorem~\ref{everywhere_R1_reduction} assumes only the existence of \emph{local} models of $X_{\Q}$ near each prime. In particular, assuming a model of $X_{\Q}$ exists globally, the result specializes to

\begin{corollary}\label{over_Z_trivial_Albanese}  
Let $X\to\Spec\Z$ be a flat projective morphism with generic fibre $X_{\Q}$ normal and geometrically connected. If every closed fibre is $R_1$, then $H^1(X_{\Q},\oo_{X_{\Q}})=0$. 
\end{corollary}

\begin{corollary}
Let $X_{\Q}$ be a normal, projective, geometrically connected scheme over $\Q$ with $H^1(X_{\Q},\oo_{X_{\Q}})\ne 0$. Then every flat projective model $X$ over $\Z$ admits a prime $p$ such that the fibre $X_p$ is singular along a divisor.
\end{corollary}

\begin{corollary}\label{no_irregular_normal_over_Z}
There is no flat projective family of irregular normal \break varieties over $\Z$.
\end{corollary}

\begin{corollary}
There is no nonzero abelian variety over $\Q$ with everywhere $R_1$ reduction.
\end{corollary}

More generally, Theorem~\ref{Main_result} systematically generalizes results for abelian varieties over a number field $K$ with everywhere semi-abelian reduction to normal projective varieties over $K$ via their Albanese varieties, replacing semi-abelian (resp.~good) reduction with $N_1$ (resp.~$R_1$) reduction.

For instance, Brumer--Kramer \cite[Thm.~1]{BK01} and Schoof \cite[Thm.~1.1]{Sch05} prove that, for $l=2,3,5,7$ or $13$, there is no nonzero abelian variety over $\Q$ with good reduction at all primes $\ne l$ and semi-abelian reduction at $l$. Theorem~\ref{Main_result} generalizes this to the following

\begin{corollary}\label{R1_everywhere_except_N1_at_a_prime} 
Let $l\in\{2,3,5,7,13\}$. If $X_{\Q}$ is a normal, projective, and geometrically connected scheme over $\Q$ with $R_1$ reduction at all primes $\ne l$ and $N_1$ reduction at $l$, then $\Alb_{X_{\Q}/\Q}$ is trivial. 
\end{corollary}

\begin{proof}
By Theorem~\ref{Main_result}, $\Alb_{X_{\Q}/\Q}$ has good reduction at all primes $\ne l$ and has semi-abelian reduction at $l$.
\end{proof}

As another example, a result of Schoof \cite[Thm.~1.2]{Sch05} for $l=11$ extends immediately to 

\begin{corollary}
If $X_{\Q}$ is a normal, projective, and geometrically connected scheme over $\Q$ with $R_1$ reduction outside $11$ and $N_1$ reduction at $11$, then $\Alb_{X_{\Q}/\Q}$ is isogenous over $\Q$ 
to a power of the Jacobian of the modular curve $X_0(11)$.
\end{corollary}

Corollary~\ref{R1_everywhere_except_N1_at_a_prime} also implies the following

\begin{corollary}\label{slc_reduction}
Let $l\in\{3,5,7,13\}$. If $X_{\Q}$ is a normal, projective, and geometrically connected scheme over $\Q$ with $R_1$ reduction at all primes $\ne l$ and semi-log canonical reduction at $l$, then $\Alb_{X_{\Q}/\Q}$ is trivial. 
\end{corollary}

\begin{proof}
By definition, a semi-log canonical scheme is either regular or nodal in the sense of Definition~\ref{intrinsic_node} at every codimension-one point, and is thus geometrically $N_1$ if $\ch\ne2$ (Lemma~\ref{intrinsic_node_is_node}).
\end{proof}

We interpret our results from the moduli stack perspective. Following Koll\'ar--Shepherd-Barron \cite{KSB88}, the KSB moduli stack $\overline{\mathscr{M}}_{2,v}$ of stable surfaces of fixed volume $v$ parametrizes families of stable surfaces (cf. \cite[Def.~(1.5.a)]{Pat17} and \cite{Kol23b}). While classical theorems of Fontaine and Abrashkin forbid the existence of $\Z$-points within the smooth irregular locus, Theorem~\ref{everywhere_R1_reduction} extends this nonexistence to the normal irregular locus. Moreover, Corollary~\ref{slc_reduction} implies that this absence of $\Z$-points persists even if one allows a genuine semi-log canonical degeneration at a specific prime $l\in\{3,5,7,13\}$---that is, even if these $\Z$-points meet the boundary of the moduli stack $\overline{\mathscr{M}}_{2,v}$ at $l$.

\begin{remark}
The conditions $R_1$ and $N_1$ extend to higher codimensions via Serre's condition $R_a$ and a corresponding notion $N_a$. The author is pursuing these extensions and studying their arithmetic applications.
\end{remark}

\subsection{Obstructions}\label{Obstructions}
Recall that the proof in \cite[Thm.~2.4]{DM69} that a curve with stable reduction has a Jacobian with semi-abelian reduction proceeds as follows: If $C_K$ admits a regular model $\mathscr{C}$ over $S$ whose special fibre $\mathscr{C}_k$ is a reduced nodal curve, then by a theorem of Raynaud \cite{Ray70} (see also \cite[Thm.~9.5/4]{BLR90}) the identity component $\mathscr{N}^0$ of the N\'eron model of its Jacobian $\Pic^0_{C_K/K}$ coincides with $\Pic^0_{\mathscr{C}/S}$, hence $\mathscr{N}^0_k\cong\Pic^0_{\mathscr{C}_k/k}$, which is a semi-abelian variety. 

The isomorphism $\mathscr{N}^0\cong\Pic^0_{\mathscr{C}/S}$ established for curves does not \mbox{extend} to higher dimensions. The main obstruction is that $\Pic_{X/S}$ may fail to be smooth over $S$ for families $X\to S$ of relative dimension at least 2. More precisely, the obstruction to the formal smoothness of $\Pic_{X/S}$ lies in $H^2(X_k,\oo_{X_k})$ (cf. \cite[Prop.~8.4/2]{BLR90}), which vanishes for curves. In practice, $\Pic_{X/S}$ can fail to be smooth when $h^1(X_k,\oo_{X_k})$ jumps, or when $\Pic^0$ of a fibre is nonreduced (which can occur in characteristic $p$ even for smooth projective surfaces over an algebraically closed field; cf. \cite{Igu55}; see also \cite{Lie09}). Neither situation arises in the curve case. Thus, whenever $\Pic_{X/S}$ is not smooth, it cannot be a N\'eron model. This failure precludes a direct generalization of the curve proof and requires a different approach. Our proof relies on the N\'eron model of $(\Pic^0_{X_K/K})_{\re}$ rather than on $\Pic_{X/S}$, so none of these phenomena obstructs the argument.

\subsection{Comparison}\label{Comparison} A higher-dimensional analogue of the implication from stable reduction to semi-abelian reduction follows from classical results of Grothendieck and Serre--Tate. If a smooth, proper, geometrically connected scheme $X_K$ has semi-stable (resp.~good) reduction, Grothendieck's theory of nearby and vanishing cycles implies that the monodromy action of the inertia group on $H^1(X_{\ol K},\Q_{\ell})$ is unipotent (resp.~trivial). Moreover, one has $H^1(X_{\ol K},\Q_{\ell})\cong H^1((\Alb_{X_K/K})_{\ol K},\Q_{\ell})$ as $\operatorname{Gal}(\ol K/K)$-representations. By Grothendieck's criterion for semi-abelian reduction of abelian varieties (resp. the N\'eron--Ogg--Shafarevich criterion), this forces $\Alb_{X_K/K}$ to have semi-abelian (resp. good) reduction.

While this classical approach provides an explicit cohomological route, it operates under strict geometric constraints: it requires a regular model $X$, a smooth generic fibre $X_K$, and a reduced normal crossings special fibre. Our geometric approach via Bertini theorems bypasses this machinery and allows us to work with significantly more general models. Assuming only that the model $X$ is projective, we impose no regularity conditions on it and allow $X_K$ to be merely geometrically normal---advantages particularly relevant in dimensions greater than three, where resolution of singularities remains open for models in mixed characteristic and varieties in positive characteristic, respectively. Crucially, our approach requires control over singularities of the special fibre only in codimension one; the special fibre may be nonreduced and possess arbitrarily bad singularities in codimension two or higher. This flexibility enables the arithmetic nonexistence results in the singular setting discussed above.

\subsection{Strategy}\label{Strategy}

The technical crux of this note is executing the inductive step, which requires proving that our chosen reduction type is preserved under taking suitable hypersurface sections. Motivated by classical Bertini theorems, we seek the weakest singularity condition on the special fibre that is inherited by such sections while still ensuring semi-abelian reduction of the Albanese variety. This leads us to formulate the notion of a \emph{node of codimension one} (Definition~\ref{definition_node}), tailored to satisfy three conditions:

\begin{enumerate}
\item It is preserved under taking appropriate hypersurface sections.
\item It specializes to the classical notion of nodes on curves.
\item It matches the notion of nodes underlying KSBA moduli theory (when the residue field is perfect of odd characteristic).
\end{enumerate}

While defining the node algebraically via completed strict henselizations is convenient for verifying conditions $(2)$ and $(3)$, proving condition $(1)$ is less straightforward. Inspired by a remark in Bosch, L\"utkebohmert, and Raynaud \cite[p.~246]{BLR90} suggesting an alternative definition of nodes via \'etale roofs, we bridge this gap by proving a geometric characterization of the node via \'etale morphisms using Artin approximation (Theorem~\ref{Geometric_characterization_of_nodes}). Since we are not aware of an explicit reference for this precise formulation, we include a proof. This geometric characterization allows us to establish our main technical result: a Bertini theorem for nodes of codimension one (Theorem~\ref{N1_Bertini_infinite_field}), thereby fulfilling condition $(1)$.

\subsection{Sketch of the proof} The proof of Theorem~\ref{Main_result} proceeds by induction on dimension. If $\dim X_K=1$, geometric normality yields smoothness, and the conclusion follows from \cite[Thm.~2.4]{DM69} (or more generally \cite[Cor.~9.7/2]{BLR90}). Suppose $\dim X_K\ge2$. If $X_K$ has geometrically $N_1$ reduction over a DVR, Theorem~\ref{geo_N1_reduction_Bertini_DVR} produces a hypersurface section $H_K\subset X_K$ of degree $d\gg0$ that also has geometrically $N_1$ reduction over the same base. Furthermore, as $X_K$ is projective and geometrically normal, we show that the natural morphism of abelian varieties over $K$
\[
(\Pic^0_{X_K/K})_{\re}\to(\Pic^0_{H_K/K})_{\re}
\]
has finite kernel for $d\gg0$. By the induction hypothesis, $(\Pic^0_{H_K/K})_{\re}$ has semi-abelian reduction, hence so do $(\Pic^0_{X_K/K})_{\re}$ and its dual $\Alb_{X_K/K}$.

\smallskip
\textbf{Notation.} Throughout this note, $S$ denotes a Dedekind scheme (i.e. a Noetherian, normal, connected scheme of dimension one) with fraction field $K$, and $s\in S$ denotes a closed point. For a scheme $X$, a \emph{codimension-one point} $x\in X$ is a point such that $\dim\oo_{X,x}=1$. We denote by $\ka(x)$ the residue field of a point $x$. A \emph{pointed \'etale morphism} $(U,u)\to(X,x)$ is an \'etale morphism $f\colon U\to X$ such that $f(u)=x$. A \emph{generic point} of a scheme $Z$ refers to the generic point of an irreducible component of $Z$. 

\subsection*{Acknowledgments} I am grateful to Michael McQuillan for an enlightening question and for helpful discussions during the early stages of this note, to K\k{e}stutis \v{C}esnavi\v{c}ius for prompt and helpful feedback on Theorem~\ref{Geometric_characterization_of_nodes}, and to J\'anos Koll\'ar for prompt and helpful comments, in particular for pointing out that an earlier question (now Corollary~\ref{SAR_iff_geo_N1}) is answered in equal characteristic $0$. I would like to thank Jungkai Alfred Chen and Ching-Jui Lai for helpful conversations and for their support. I thank my family, especially my wife, Wei-Na, for their continued support.

This work is supported by the National Science and Technology Council of Taiwan (grant numbers 114-2811-M-006-030 and 114-2639-M-002-009-ASP), and was partially supported by the National Science Foundation (grant number DMS-1802460) and a grant from the Simons Foundation.

\section{Nodes}

In this section, we define the notion of a \emph{node of codimension one} (Definition~\ref{definition_node}) and compare it with the corresponding notions in the Stacks Project (Definition~\ref{Stacks_node}) and in Koll\'ar's book~\cite{Kol13} (Definition~\ref{intrinsic_node}).

\medskip
Let $X$ be a scheme locally of finite type over a field $k$, and let $x\in X$ be a codimension-one point (i.e. $\dim\oo_{X,x}=1$). 

\begin{definition}\label{definition_node}
A point $x\in X$ is a \emph{node} if its residue field $\ka(x)$ is separable over $k$ and, for some (equivalently any) separable algebraic closure $\ka(x)\hookrightarrow\ka(x)^{\sep}$, there is a $k$-algebra isomorphism of completed strict henselizations
\[
\widehat{\oo_{X,x}^{\sh}}\cong\ka(x)^{\sep}[\![a,b]\!]/(ab)
\]
where the $k$-algebra structure on the right is induced by $k\hookrightarrow\ka(x)\hookrightarrow\ka(x)^{\sep}$. 
\end{definition}

\begin{remark}
The separability of $\ka(x)/k$ in Definition~\ref{definition_node} is necessary for the existence of an \'etale roof appearing in a geometric characterization of nodes (Theorem~\ref{Geometric_characterization_of_nodes}). This requirement is also consistent with the treatment of nodes of curves in \cite[\href{https://stacks.math.columbia.edu/tag/0C4D}{Tag 0C4D}]{stacks-project}, where the residue fields are required to be separable over the base field.
\end{remark}

\begin{remark} We relate Definition~\ref{definition_node} to standard notions of nodes in the literature that are relevant to our proofs.

\begin{enumerate}
    \item When $\dim X=1$ and $k$ is algebraically closed, Definition~\ref{definition_node} matches the classical notion of an ordinary double point, such as the one used in \cite[p.~246, following Def.~9.2/6]{BLR90}.  
    \item  When $\dim X=1$ and $k$ is any field, Definition~\ref{definition_node} coincides with nodes defined in the Stacks Project \cite[\href{https://stacks.math.columbia.edu/tag/0C47}{Tag 0C47}]{stacks-project}; see Definition~\ref{Stacks_node} and Lemma~\ref{stacks_node_is_node}. 
    \item  In any dimension, when $\ch k\ne2$, Definition~\ref{definition_node} is equivalent to the notion of nodes in KSBA moduli theory \cite[Item~1.41]{Kol13} precisely when $\ka(x)/k$ is separable; see Definition~\ref{intrinsic_node} and Lemma~\ref{intrinsic_node_is_node}. When $\ch k=2$, our definition is strictly stronger even when $\ka(x)/k$ is separable (see Example~\ref{example_Whitney_umbrella}).
\end{enumerate}
\end{remark}

\subsection{Stacks node}

\begin{definition}[cf. {\cite[\href{https://stacks.math.columbia.edu/tag/0C47}{Tag 0C47}]{stacks-project}}]\label{Stacks_node}
Let $X$ be a 1-dimensional scheme locally of finite type over a field $k$. A closed point $x\in X$ is a \emph{Stacks node} if for an algebraic closure $k\subset\ol k$ there exists a point $\ol x\in\ol X:=X\times_k\ol k$ mapping to $x$ such that $\widehat{\oo_{\ol X,\ol x}}\cong\ol k[\![a,b]\!]/(ab)$.
\end{definition}

For a list of equivalent conditions of Stacks node (Definition~\ref{Stacks_node}), see \cite[\href{https://stacks.math.columbia.edu/tag/0C4D}{Tag 0C4D}]{stacks-project}.

\begin{lemma}\label{stacks_node_is_node}
Let $X$ be a scheme locally of finite type over a field $k$. Then a point $x\in X$ is a node (Definition~\ref{definition_node}) if and only if $\ka(x)/k$ is separable, $\oo_{X,x}$ is 1-dimensional, reduced, has $\delta$-invariant 1, and has 2 geometric branches (i.e. it verifies item~(6) of \cite[\href{https://stacks.math.columbia.edu/tag/0C4D}{Tag 0C4D}]{stacks-project}). In particular, when $\dim X=1$, a closed point $x\in X$ is a node (Definition~\ref{definition_node}) if and only if it is a Stacks node (Definition~\ref{Stacks_node}).
\end{lemma}

\begin{proof}
Note that $\oo_{X,x}$ is Nagata since $X$ is locally of finite type over a field \cite[\href{https://stacks.math.columbia.edu/tag/035B}{Tag 035B}]{stacks-project}.

Let $x\in X$ be a node; by definition, $\ka(x)/k$ is separable. Since $\widehat{\oo_{X,x}^{\sh}}$ is 1-dimensional, so is $\oo_{X,x}$ \cite[\href{https://stacks.math.columbia.edu/tag/07NV}{Tag 07NV}, \href{https://stacks.math.columbia.edu/tag/06LK}{Tag 06LK}]{stacks-project}. Because $\widehat{\oo_{X,x}^{\sh}}\cong\ka(x)^{\sep}[\![a,b]\!]/(ab)$ is reduced, so is $\oo_{X,x}^{\sh}$ \cite[\href{https://stacks.math.columbia.edu/tag/07NZ}{Tag 07NZ}]{stacks-project} and hence $\oo_{X,x}$ \cite[\href{https://stacks.math.columbia.edu/tag/06DH}{Tag 06DH}]{stacks-project}. As $\widehat{\oo_{X,x}^{\sh}}$ has $\delta$-invariant 1, so does $\oo_{X,x}$ \cite[\href{https://stacks.math.columbia.edu/tag/0C3W}{Tag 0C3W}]{stacks-project}. Since $\widehat{\oo_{X,x}^{\sh}}$ has 2 geometric branches, so do $\oo_{X,x}^{\sh}$ \cite[\href{https://stacks.math.columbia.edu/tag/0C2D}{Tag 0C2D}]{stacks-project} and $\oo_{X,x}$ by definition \cite[\href{https://stacks.math.columbia.edu/tag/0C26}{Tag 0C26}]{stacks-project}.

Conversely, since $\ka(x)/k$ is separable, so is $\ka(x)^{\sep}/k$; since $\oo_{X,x}$ is a reduced Nagata local ring of dimension 1, so is $\widehat{\oo_{X,x}^{\sh}}$ \cite[\href{https://stacks.math.columbia.edu/tag/0C3V}{Tag 0C3V}]{stacks-project}. 
Since $\oo_{X,x}$ has $\delta$-invariant 1 and has 2 geometric branches, so does $\widehat{\oo_{X,x}^{\sh}}$. Applying \cite[\href{https://stacks.math.columbia.edu/tag/0C4A}{Tag 0C4A}]{stacks-project} to $\widehat{\oo_{X,x}^{\sh}}$ yields that it satisfies condition (3) of \cite[\href{https://stacks.math.columbia.edu/tag/0C49}{Tag 0C49}]{stacks-project}, namely, there exists a $k$-algebra isomorphism $\widehat{\oo_{X,x}^{\sh}}\cong\ka(x)^{\sep}[\![u,v]\!]/(q)$ with $q=\alpha u^2+\beta uv+\gamma v^2$ a nondegenerate quadratic form over $\ka(x)^{\sep}$. It remains to show that every such $q$ is equivalent to $uv$ over $\ka(x)^{\sep}$.

Nondegeneracy of $q$ means $\beta^2-4\alpha\gamma\ne0$. If $\ch\ka(x)^{\sep}\ne2$, $q$ factors into two distinct linear factors as $\ka(x)^{\sep}$ contains all square roots; hence it is equivalent to $uv$. Suppose $\ch\ka(x)^{\sep}=2$. The polynomial $\alpha T^2+\beta T+\gamma$ has derivative $\beta\ne 0$, and is therefore separable. Its roots $z_1\neq z_2$ lie in $\ka(x)^{\sep}$. Thus, if $\alpha\ne0$, $q=\alpha(u-z_1 v)(u-z_2 v)$; if $\alpha=0$, $q=v(\beta u+\gamma v)$. In both cases $q$ factors into two linearly independent linear forms over $\ka(x)^{\sep}$ and hence is equivalent to $uv$. This completes the proof.
\end{proof}

\subsection{Intrinsic node}

\begin{definition}[cf. {\cite[1.41]{Kol13}}; see also {\cite[Def.~3.1]{Tan16}}, {\cite[Def.~2.1]{Pos24}}]\label{intrinsic_node}
Let $X$ be a scheme. A point $x\in X$ is an \emph{intrinsic node} (called simply a \emph{node} in \cite[1.41]{Kol13}) if there exists a ring isomorphism $\oo_{X,x}\cong R/(f)$, where $(R,\m,\ka)$ is a 2-dimensional regular local ring, $f\in\m^2$, and its image in $\m^2/\m^3$ is not a square up to a unit in $\ka$.
\end{definition}

\begin{remark}
The formulation in \cite[1.41]{Kol13} requires that $f$ not be a square in $\m^2/\m^3$. This condition, when working over fields that are not algebraically closed, is naturally understood to mean that the image of $f$ in $\m^2/\m^3$ is not a square up to a unit in $\ka$. For instance, $f=-x^2\in\R[x,y]_{(x,y)}$ is not a square in $(x,y)^2/(x,y)^3$, but $\R[x,y]_{(x,y)}/(x^2)$ does not define the intended notion of a node.
\end{remark}

\begin{lemma}\label{intrinsic_node_is_node}
Let $X$ be a scheme locally of finite type over a field $k$ and $x\in X$ a codimension-one point. If $x\in X$ is a node (Definition~\ref{definition_node}), then it is an intrinsic node (Definition~\ref{intrinsic_node}). The converse holds if $\ch k\ne2$ and $\ka(x)/k$ is separable.
\end{lemma}

\begin{proof}
Assume $x\in X$ is a node. Fix a separable algebraic closure $\ka(x)\hookrightarrow K$. Since $\widehat{\oo_{X,x}^{\sh}}\cong K[\![u,v]\!]/(uv)$, its maximal ideal is generated by two elements. Since strict henselization \cite[\href{https://stacks.math.columbia.edu/tag/07QM}{Tag 07QM}]{stacks-project} and completion preserve $\m/\m^2$, the maximal ideal of $\oo_{X,x}$ is generated by two elements. Because $\oo_{X,x}$ is a localization of a finitely generated $k$-algebra, it can be written as $R/I$, where $R$ is a regular local ring of dimension 2. As $\widehat{\oo_{X,x}^{\sh}}$ is reduced of dimension 1, so is $\oo_{X,x}$. Since $R$ is a two-dimensional UFD and $R/I$ is reduced of dimension 1, the height-one radical ideal $I$ is principal; put $I=(f)$. By the Cohen structure theorem $\h R\cong\ka(x)[\![a,b]\!]$, so $\widehat{R/(f)}\cong\h R/(f)\cong\ka(x)[\![a,b]\!]/(f)$. Thus passing $\oo_{X,x}\cong R/(f)$ to the completed strict henselizations yields $K[\![u,v]\!]/(uv)\cong K[\![a,b]\!]/(f)$. This induces an isomorphism of the associated graded rings 
\[
K[u,v]/(uv)\cong K[a,b]/(\operatorname{in}(f)),
\]
where $\operatorname{in}(f)$ denotes the lowest-degree homogeneous component of $f$. It follows that $\operatorname{in}(f)$ has degree 2, and is not a square up to a unit since $K[u,v]/(uv)$ is reduced. Because $\operatorname{in}(f)$ represents the image of $f$ in $\m^2/\m^3$, where $\m$ denotes the maximal ideal of $R$, the assertion follows.

\smallskip
Conversely, by Definition~\ref{intrinsic_node} we can write $\oo_{X,x}\cong R/(f)$ where $(R,\m,\ka)$ is a regular local ring of dimension 2, $f\in\m^2$, and the image of $f$ in $\m^2/\m^3$ is not a square up to a unit in $\ka$. A priori, this isomorphism may not be chosen as a $k$-algebra isomorphism. To construct one, note that these conditions imply that the maximal ideal $\m_x$ of $\oo_{X,x}$ can be generated by two elements (as $\m_x/\m_x^2\cong\m/\m^2$). Therefore, as $\ka(x)/k$ is separable, by (the proof of) \cite[\href{https://stacks.math.columbia.edu/tag/0C52}{Tag 0C52}]{stacks-project} we find a $k$-algebra isomorphism $\ka(x)[\![u,v]\!]/I\cong\h\oo_{X,x}$ for some ideal $I$. As $\oo_{X,x}$ is reduced, $\h\oo_{X,x}$ is reduced \cite[\href{https://stacks.math.columbia.edu/tag/07NZ}{Tag 07NZ}]{stacks-project}. Because $\ka(x)[\![u,v]\!]$ is a two-dimensional UFD and $\ka(x)[\![u,v]\!]/I$ is reduced of dimension 1, the height-one radical ideal $I$ is principal; put $I=(g)$. On the other hand, since $\h R\cong\ka(x)[\![a,b]\!]$, we have $\widehat{R/(f)}\cong\ka(x)[\![a,b]\!]/(f)$. Thus the local ring isomorphisms $\ka(x)[\![u,v]\!]/(g)\cong\h\oo_{X,x}\cong\ka(x)[\![a,b]\!]/(f)$ induce an isomorphism of the associated graded rings 
\begin{equation}\label{graded}
\ka(x)[u,v]/(\operatorname{in}(g))\cong \ka(x)[a,b]/(\operatorname{in}(f)).    
\end{equation}  
As $\operatorname{in}(f)$ has degree 2, $\operatorname{in}(g)$ has degree 2. Fix a separable algebraic closure $\ka(x)\hookrightarrow K$. Since $K$ is separably closed and $\ch K\ne2$, $\operatorname{in}(g)=L_1 L_2$ for some linear forms $L_1, L_2\in K[u,v]$. Suppose $L_1=\lambda L_2$ for some $\lambda\in K^{\times}$. Writing $L_2=\alpha u+\beta v$ with $\alpha\ne0$, we find that $\operatorname{in}(g)=(\lambda\alpha^2)(u+\frac{\beta}{\alpha}v)^2$ is a square up to a unit already in $\ka(x)[u,v]$, which by (\ref{graded}) implies that $\operatorname{in}(f)$ is a square up to a unit, contradicting the condition on $f$. Thus $L_1, L_2$ are linearly independent. A standard induction on homogeneous degree then yields a factorization $g=g_1g_2$ in $K[\![u,v]\!]$ such that $g_1\equiv L_1\mod{(u,v)^2}$ and $g_2\equiv L_2\mod{(u,v)^2}$. Since the linear parts of $g_1,g_2$ are linearly independent, the automorphism of $K[\![u,v]\!]$ determined by $u\mapsto g_1$ and $v\mapsto g_2$ induces an isomorphism $K[\![u,v]\!]/(uv)\cong K[\![u,v]\!]/(g)$. On the other hand, the $k$-algebra isomorphism $\ka(x)[\![u,v]\!]/(g)\cong\h\oo_{X,x}$ yields $K[\![u,v]\!]/(g)\cong\widehat{\oo_{X,x}^{\sh}}$ as $k$-algebras (Lemma~\ref{completed_strict_Hensel}). Hence the conclusion follows. 
\end{proof}

The following is a typical example of an intrinsic node (Definition~\ref{intrinsic_node}) that is not a node (Definition~\ref{definition_node}); see for example \cite[Ex.~3.9]{Pos24}.

\begin{example}[The Whitney umbrella]\label{example_Whitney_umbrella}
Let $k$ be a field of characteristic 2. Let $X=\Spec k[u,v,t]/(u^2-tv^2)$ and $x=(u,v)\in X$. Then $\oo_{X,x}\cong k(t)[u,v]_{(u,v)}/(u^2-tv^2)$. Since $u^2-tv^2$ is not a square modulo $(u,v)^3$, $x\in X$ is an intrinsic node. However, since $\ch k=2$, $\sqrt{t}\not\in k(t)^{\sep}$ and $u^2-tv^2$ is irreducible over $k(t)^{\sep}$. Thus $\widehat{\oo_{X,x}^{\sh}}$ has only one minimal prime, so $x\in X$ is not a node.
\end{example}

We note that this distinction persists even for curves. Indeed, being an intrinsic node is an intrinsic property of the local ring, whereas being a node is a property relative to the base field.

\begin{example}
Let $k'/k$ be a finite inseparable field extension. Consider $X=\Spec k'[u,v]/(uv)$ and $x=(u,v)$. Then $x\in X$ is an intrinsic node. On the other hand, for $x\in X$ to be a node, $\ka(x)\cong k'$ must be separable over $k$. Thus, $x\in X$ is a node when $X$ is viewed as an algebraic $k'$-scheme, but it is not a node when $X$ is viewed as an algebraic $k$-scheme.
\end{example}

\section{Geometric characterization of nodes}

The main result of this section is Theorem~\ref{Geometric_characterization_of_nodes}.

In higher dimensions, a natural attempt to define a node of codimension one is to mimic the classical definition for curves by requiring the singularity to split after a base field extension $L/k$ (cf. Definition~\ref{Stacks_node}), so that the completed local ring of the base change takes the standard split form $\widehat{\oo_{X_L,x_L}}\cong\ka(x_L)[\![a,b]\!]/(ab)$. However, this naive approach fails because the obstruction to splitting the node can live in the residue field of the point, rather than the base field.

\begin{example}[Base change fails to split nodes in higher dimensions]\label{Base_change_fails_to_split_nodes} 
Assume $\ch k\ne2$ and let $X=\Spec k[u,v,t]/(u^2-tv^2)$ be the scheme from Example~\ref{example_Whitney_umbrella}. The point $x=(u,v)\in X$ is a node (Definition~\ref{definition_node}). For any field extension $L/k$, let $x_L=(u,v)\in X_L=X\times_kL$ denote the point lying over $x$. Then $\widehat{\oo_{X_L,x_L}}\cong L(t)[\![u,v]\!]/(u^2-tv^2)$. This ring is not isomorphic to $L(t)[\![u,v]\!]/(uv)$ as $u^2-tv^2$ remains irreducible in $L(t)[\![u,v]\!]$ for all field extensions $L/k$.
\end{example}

To bypass this obstruction, one needs to allow residue field extensions. Following \cite[p.246, following Def.~9.2/6]{BLR90}, a node on a curve over a field $k$ can be formulated by the existence of an \'etale neighborhood which is \'etale over the union of coordinate axes of $\mathbb{A}^2_k$. This geometric formulation extends naturally to higher dimensions and makes preservation under hypersurface sections more tractable. While the existence of an \'etale roof readily implies isomorphic (completed) strict henselizations of the local rings, the converse is less obvious; Theorem~\ref{Geometric_characterization_of_nodes} bridges this gap. Since we are not aware of an explicit reference applying Artin approximation to completed strict henselizations in this setting, we include a proof for completeness.

To reduce Theorem~\ref{Geometric_characterization_of_nodes} to Artin approximation, we need all constructed isomorphisms to respect the $k$-algebra structure, which is ensured by the separability of the residue field extension $\ka(x)/k$. 

We begin with a few lemmas.

\begin{lemma}\label{completed_strict_Hensel}
Let $X$ be a scheme locally of finite type over a field $k$, and let $x\in X$ be a point. Fix a separable algebraic closure $\ka(x)\hookrightarrow\ka(x)^{\sep}$. Then a $k$-algebra isomorphism $\widehat{\oo_{X,x}}\cong\ka(x)[\![a,b]\!]/(g)$ induces a $k$-algebra isomorphism $\widehat{\oo_{X,x}^{\sh}}\cong\ka(x)^{\sep}[\![a,b]\!]/(g)$.
\end{lemma}

\begin{proof}
The $k$-algebra isomorphism $\widehat{\oo_{X,x}}\cong\ka(x)[\![a,b]\!]/(g)$ induces a $k$-algebra section $\iota_x\colon\ka(x)\hookrightarrow\widehat{\oo_{X,x}}$. View $\ka(x)^{\sep}$ as a $k$-algebra via $k\hookrightarrow\ka(x)\hookrightarrow\ka(x)^{\sep}$. Applying \cite[\href{https://stacks.math.columbia.edu/tag/0C2Z}{Tag 0C2Z}]{stacks-project} to $\iota_x\colon\ka(x)\hookrightarrow\widehat{\oo_{X,x}}$ with respect to their maximal ideals yields $(\widehat{\oo_{X,x}})^{\sh}\cong\widehat{\oo_{X,x}}\otimes_{\ka(x)}\ka(x)^{\sep}$ (here, $\widehat{\oo_{X,x}}$ is henselian as it is complete local). Complete both sides with respect to their maximal ideals. The completion of the left-hand side is $\widehat{\oo_{X,x}^{\sh}}$ by item~(f) of \cite[\href{https://stacks.math.columbia.edu/tag/06LJ}{Tag 06LJ}]{stacks-project}. The right-hand side is isomorphic to $\ka(x)[\![a,b]\!]/(g)\otimes_{\ka(x)}\ka(x)^{\sep}$, whose completion is $\ka(x)^{\sep}[\![a,b]\!]/(g)$. The resulting isomorphism respects the $k$-algebra structure by construction.
\end{proof}

\begin{lemma}\label{O_u_tensor_product} 
Let $X$ be a scheme locally of finite type over a field $k$. Suppose $x\in X$ admits a $k$-algebra section $\iota_x\colon\ka(x)\hookrightarrow\widehat{\oo_{X,x}}$ (for instance, if $\ka(x)/k$ is separable). Then for any pointed \'etale $k$-morphism $(U,u)\to(X,x)$, $\widehat{\oo_{U,u}}$ (endowed with the $\ka(x)$-algebra structure induced by $\iota_x$) admits a $\ka(x)$-algebra section $\ka(u)\hookrightarrow\widehat{\oo_{U,u}}$ such that the natural map 
\[
\widehat{\oo_{X,x}}\otimes_{\ka(x)}\ka(u)\to\widehat{\oo_{U,u}}
\]
is an isomorphism of $k$-algebras.
\end{lemma}

\begin{proof}
This is \cite[Ex.~III.10.4]{Har77}; we restate it here to clarify our use. Since $\ka(u)/\ka(x)$ is separable, \cite[\href{https://stacks.math.columbia.edu/tag/0C34}{Tag 0C34}]{stacks-project} (or \cite[Thm.~II.8.25A]{Har77}) yields a $\ka(x)$-algebra section $\ka(u)\hookrightarrow\widehat{\oo_{U,u}}$. By construction all the maps respect the $k$-algebra structures, hence so does the natural map. 
\end{proof}

\begin{lemma}\label{etale_nbd}
Let $X$ be a scheme locally of finite type over a field $k$, and let $x\in X$ be a codimension-one point. The following are equivalent: 

\begin{enumerate}
    \item[$(C)$] $x\in X$ is a node (Definition~\ref{definition_node}); i.e. $\ka(x)/k$ is separable and $\widehat{\oo_{X,x}^{\sh}}\cong\ka(x)^{\sep}[\![\alpha,\beta]\!]/(\alpha\beta)$ as $k$-algebras.
    \item[$(D)$] $\ka(x)/k$ is separable and there exists a pointed \'etale $k$-morphism $(U,u)\to (X,x)$ such that $\widehat{\oo_{U,u}}\cong\ka(u)[\![\alpha,\beta]\!]/(\alpha\beta)$ as $k$-algebras.
\end{enumerate}
\end{lemma}

\begin{proof}
$(D)\Rightarrow(C)$: Since $U\to X$ is \'etale, by \cite[\href{https://stacks.math.columbia.edu/tag/094Z}{Tag 094Z}]{stacks-project} we have $\widehat{\oo_{X,x}^{\sh}}\cong\widehat{\oo_{U,u}^{\sh}}$ as $k$-algebras. Applying Lemma~\ref{completed_strict_Hensel} to $u\in U$ gives $\widehat{\oo_{U,u}^{\sh}}\cong\ka(x)^{\sep}[\![\alpha,\beta]\!]/(\alpha\beta)$ as $k$-algebras, so $(C)$ follows. 

\smallskip
$(C)\Rightarrow(D)$: The ring $\oo_{X,x}$ is Nagata (\cite[\href{https://stacks.math.columbia.edu/tag/035B}{Tag 035B}]{stacks-project}) with $\ka(x)/k$ separable and, by Lemma~\ref{stacks_node_is_node}, is reduced of dimension 1, has $\delta$-invariant 1, and has 2 geometric branches. Thus $\oo_{X,x}$ satisfies the equivalent conditions of \cite[\href{https://stacks.math.columbia.edu/tag/0C4A}{Tag 0C4A}]{stacks-project}. Condition~(3) of \cite[\href{https://stacks.math.columbia.edu/tag/0C49}{Tag 0C49}]{stacks-project} yields a $k$-algebra isomorphism 
\begin{equation}\label{xxabc}
  \widehat{\oo_{X,x}}\cong\ka(x)[\![t,s]\!]/(at^2+bts+cs^2)    
\end{equation}
where $at^2+bts+cs^2$ is a nondegenerate quadratic form over $\ka(x)$. Thus it splits in $\ka(x)^{\sep}[t,s]$ with distinct linear factors: $at^2+bts+cs^2=(a_1t-b_1s)(a_2t-b_2s)\in\ka(x)^{\sep}[t,s]$ with $a_2b_1-a_1b_2\ne0$.

Let $L=\ka(x)(a_1,a_2,b_1,b_2)\subset\ka(x)^{\sep}$. So $L/\ka(x)$ is finite separable. By \cite[\href{https://stacks.math.columbia.edu/tag/02LF}{Tag 02LF}]{stacks-project} there exists an \'etale morphism $(U,u)\to(X,x)$ such that $\ka(u)\cong L$ as $\ka(x)$-algebras; we view $U$ as a $k$-scheme via $U\to X\to\Spec k$. The $k$-algebra isomorphism (\ref{xxabc}) induces a $k$-algebra section $\iota_x\colon\ka(x)\hookrightarrow\widehat{\oo_{X,x}}$. Lemma~\ref{O_u_tensor_product} applied to $\iota_x$, together with tensoring (\ref{xxabc}) with $\ka(u)$, gives $k$-algebra isomorphisms 
\[
\widehat{\oo_{U,u}}\cong\widehat{\oo_{X,x}}\otimes_{\ka(x)}\ka(u)\cong\ka(u)[\![t,s]\!]/(at^2+bts+cs^2).
\]
By construction $at^2+bts+cs^2=(a_1t-b_1s)(a_2t-b_2s)\in\ka(u)[t,s]$ with nonzero determinant $a_2b_1-a_1b_2\ne0$, so the change of coordinates $\alpha=a_1t-b_1s$ and $\beta=a_2t-b_2s$ of $\ka(u)[\![t,s]\!]$ induces a $k$-algebra isomorphism $\widehat{\oo_{U,u}}\cong\ka(u)[\![\alpha,\beta]\!]/(\alpha\beta).$
\end{proof}

\begin{lemma}\label{model_for_Artin_approx}
Let $U$ be a scheme locally of finite type over a field $k$ and $u\in U$ a codimension-one point with residue field $\ka(u)$. Assume that $\ka(u)/k$ is separable, and that there is a $k$-algebra isomorphism 
\begin{equation}\label{O_u}
\widehat{\oo_{U,u}}\cong\ka(u)[\![a,b]\!]/(ab).    
\end{equation}
Then there exists a scheme $Z=\Spec A[a,b]/(ab)$ of finite type over $k$ with a point $z=(a,b)\in Z$ such that there is a $k$-algebra isomorphism 
\[
\widehat{\oo_{U,u}}\cong\widehat{\oo_{Z,z}}.
\]
Moreover, $\Spec A$ is \'etale over some affine space $\A^n_k$, where $n=\operatorname{tr.deg}_k\ka(u)$.
\end{lemma}

\begin{proof}
Let $K=\ka(u)$. Since $K/k$ is a finitely generated field extension, we may choose a finitely generated $k$-subalgebra $A\subset K$ whose fraction field is $K$. The generic point $\eta$ of $\Spec A$ has residue field $K$, which is separable over $k$ by assumption, thus $A$ is smooth at $\eta$ over $k$ \cite[\href{https://stacks.math.columbia.edu/tag/00TV}{Tag 00TV}]{stacks-project}. Possibly after shrinking, we may assume $\Spec A$ is \'etale over some affine space $\A^n_k$ \cite[\href{https://stacks.math.columbia.edu/tag/054L}{Tag 054L}]{stacks-project}. Now define $Z=\Spec A[a,b]/(ab)$, and let $z=(a,b)\in Z$. Then $\oo_{Z,z}\cong K[a,b]_{(a,b)}/(ab)$ as $k$-algebras; passing to the completion yields a $k$-algebra isomorphism \begin{equation}\label{O_z}
\widehat{\oo_{Z,z}}\cong K[\![a,b]\!]/(ab).
\end{equation} 
The assertion follows by combining (\ref{O_u}) and (\ref{O_z}). 
\end{proof}

The following theorem establishes an equivalence between the algebraic definition of a node of codimension one and its geometric formulation via \'etale morphisms.

\begin{theorem}\label{Geometric_characterization_of_nodes} 
Let $X$ be a scheme locally of finite type over a field $k$, and let $x\in X$ be a codimension-one point. The following are equivalent.

\begin{enumerate}
    \item[$(A)$] $x\in X$ is a node (Definition~\ref{definition_node}); i.e. $\ka(x)/k$ is separable and $\widehat{\oo_{X,x}^{\sh}}\cong\ka(x)^{\sep}[\![a,b]\!]/(ab)$ as $k$-algebras. 
    \item[$(B)$] There exist pointed \'etale $k$-morphisms 
    \[
    (X,x)\leftarrow(V,v)\rightarrow(Y,y)
    \] 
    where $Y=\Spec k[a,b,t_1,\dots,t_n]/(ab)$ and $y=(a,b)\in Y$.
\end{enumerate}
\end{theorem}

\begin{proof}
$(A)\Rightarrow(B)$: By Lemma~\ref{etale_nbd}, there exists a pointed \'etale morphism $(U,u)\to (X,x)$ such that $\widehat{\oo_{U,u}}\cong\ka(u)[\![a,b]\!]/(ab)$ as $k$-algebras. Since $\ka(u)/\ka(x)$ and $\ka(x)/k$ are separable, $\ka(u)/k$ is separable. Thus Lemma~\ref{model_for_Artin_approx} yields a scheme $Z=\Spec A[a,b]/(ab)$ of finite type over $k$, where $\Spec A$ is \'etale over $\A^n_k$, and a point $z=(a,b)\in Z$ such that there is a $k$-algebra isomorphism $\widehat{\oo_{U,u}}\cong\widehat{\oo_{Z,z}}.$ Applying Artin approximation in the form \cite[\href{https://stacks.math.columbia.edu/tag/0CAV}{Tag 0CAV}]{stacks-project} (cf. \cite[Cor.~2.6]{Art69}), we find a common \emph{elementary} \'etale neighborhood $V$, namely a pointed scheme $(V,v)$ and pointed \'etale $k$-morphisms 
\[
(U,u)\leftarrow(V,v)\rightarrow(Z,z)
\]
with $\ka(u)\cong\ka(v)\cong\ka(z)$. Since $\Spec A$ is \'etale over $\A^n_k$, its base change 
\[
Z\cong\Spec A\times_k\Spec k[a,b]/(ab)\to\A^n_k\times_k\Spec k[a,b]/(ab)\cong Y
\]
is also \'etale. Moreover, this morphism maps $z$ to $y$. Indeed, because $\Spec A\to\A^n_k$ is \'etale, the induced map $B:=k[t_1,\dots,t_n]\to A$ is flat and, since $B$ is a domain and $A\ne0$, injective. Thus the induced ring map $B[a,b]/(ab)\to A[a,b]/(ab)$ contracts the ideal $(a,b)$ to $(a,b)$, i.e. $Z\to Y$ maps $z$ to $y$. Therefore, the composition gives the \'etale roof:
\[
(X,x)\leftarrow(U,u)\leftarrow(V,v)\rightarrow(Z,z)\rightarrow(Y,y).
\]

$(B)\Rightarrow(A)$: Since $\ka(v)$ is separable over $\ka(y)$ and $\ka(y)\cong k(t_1,\dots,t_n)$ is separable over $k$, $\ka(v)$ and hence $\ka(x)$ is separable over $k$. The \'etale morphisms $V\to X$ and $V\to Y$ induce $k$-algebra isomorphisms $\widehat{\oo_{X,x}^{\sh}}\cong\widehat{\oo_{V,v}^{\sh}}\cong\widehat{\oo_{Y,y}^{\sh}}$ \cite[\href{https://stacks.math.columbia.edu/tag/094Z}{Tag 094Z}]{stacks-project}. Applying Lemma~\ref{completed_strict_Hensel} to $y\in Y$ gives $\widehat{\oo_{Y,y}^{\sh}}\cong\ka(x)^{\sep}[\![a,b]\!]/(ab)$, hence $(A)$ follows.
\end{proof}

\begin{remark}
Likewise, one can characterize a \emph{split node}: $\ka(x)/k$ is separable and the completion $\widehat{\oo_{X,x}}$ is isomorphic as $k$-algebras to $\ka(x)[\![a,b]\!]/(ab)$ if and only if there exist pointed \'etale $k$-morphisms $(X,x)\leftarrow(V,v)\rightarrow(Y,y)$ with $\ka(x)\cong\ka(v)$, where $(Y,y)$ is defined as in Theorem~\ref{Geometric_characterization_of_nodes}.
\end{remark}

\section{Geometric \texorpdfstring{$N_1$}{N1} and its Bertini theorems}

In this section, we define the geometric $N_1$ property and prove its Bertini theorems over infinite fields (Theorem~\ref{N1_Bertini_infinite_field}) and finite fields (Theorem~\ref{N1_Bertini_finite_field}).

\subsection{Geometric \texorpdfstring{$N_1$}{N1}}

Let $X$ be an equidimensional scheme locally of finite type over a field $k$. We say that $X$ is \emph{geometrically $R_1$} if every codimension-one point of $X$ is geometrically regular. Equivalently, the morphism $X\to\Spec k$ is smooth at every codimension-one point of $X$ \cite[\href{https://stacks.math.columbia.edu/tag/038X}{Tag 038X}]{stacks-project}.

\begin{definition}\label{Def_geo_N1}
We say that $X$ is \emph{geometrically $N_1$} (resp.~$N_1$) if every codimension-one point of $X$ is either a geometrically regular (resp.~regular) point or a node (Definition~\ref{definition_node}).
\end{definition}

\begin{lemma}\label{snc_implies_geo_N1}
Let $k$ be a field. The standard SNC scheme
\[
X=\Spec k[x_1,x_2,\dots, x_m]/(x_1x_2\cdots x_r)
\]
is geometrically $N_1$.
\end{lemma}

\begin{proof}
Let $A:=k[x_1,x_2,\dots, x_m]/(x_1x_2\cdots x_r)$. Each codimension-one point $x\in X$ corresponds to a prime $\p\subset A$ of height 1. Its preimage $P$ in $k[x_1,x_2,\dots, x_m]$ is a prime and contains $x_1x_2\cdots x_r$, so $x_i\in P$ for some $1\le i\le r$. Let $I=\{\,i\in\{1,\cdots,r\}\mid x_i\in P\,\}$. Since $\p$ has height 1, $P$ has height 2, hence $1\le|I|\le2$. 

Case 1: $|I|=1$. We may assume $I=\{1\}$. In the localization $A_{\p}$, $x_2,\dots,x_r$ are units, so $\oo_{X,x}\cong (k[x_1,x_2,\dots, x_m]/(x_1))_{\p}$ is a localization of $k[x_2,\dots, x_m]$ and thus is smooth over $k$.

Case 2: $|I|=2$. We may assume $I=\{1,2\}$. Since $P$ has height 2, $P=(x_1,x_2)$. Since $\ka(x)\cong k(x_3,\dots, x_m)$ is separable over $k$ and $\widehat{\oo_{X,x}^{\sh}}\cong \ka(x)^{\sep}[\![x_1,x_2]\!]/(x_1x_2)$ as $k$-algebras, $x\in X$ is a node.
\end{proof}

\begin{lemma}\label{geo_N1_implies_base_change_N1}
Let $X$ be a 1-dimensional scheme locally of finite type over a field $k$. If $X$ is geometrically $N_1$, then $X_{\ol k}:=X\times_k\ol k$ is $N_1$. 
\end{lemma}

\begin{proof}
Let $\ol x\in X_{\ol k}$ be a closed point. Since $\ol k/k$ is integral, its base change $X_{\ol k}\to X$ is integral and hence closed \cite[\href{https://stacks.math.columbia.edu/tag/01WM}{Tag 01WM}]{stacks-project}; thus the image $x\in X$ of $\ol x$ is a closed point. Because $X$ is geometrically $N_1$, $x$ is either geometrically regular, in which case $\ol x$ is a regular point, or a node, in which case $\ol x$ is a node by item~(3) of \cite[\href{https://stacks.math.columbia.edu/tag/0C4D}{Tag 0C4D}]{stacks-project}.
\end{proof}

\subsection{Hypersurface sections}

Let $k$ be a field and let $d\ge1$ be an integer. We identify degree $d$ hypersurfaces $H\subset\pp^N_k$ with the $k$-points of the projective space $\pp(H^0(\pp^N_k,\oo_{\pp^N_k}(d)))$. For a subscheme $X\subseteq\pp^N_k$, the notation $X\cap H$ denotes the scheme-theoretic intersection.

\subsection{A Bertini criterion for \texorpdfstring{$R_a$}{Ra}}

Let $X$ be a locally Noetherian scheme. For integers $a,b\ge0$, we say that $X$ is $R_a$ (resp.~$S_b$) if $X$ satisfies Serre's condition $R_a$ (resp.~$S_b$); cf. \cite[\href{https://stacks.math.columbia.edu/tag/033P}{Tag 033P}]{stacks-project}. When $X$ is excellent (e.g. any scheme locally of finite type over a field), its regular locus is an open subscheme, denoted by $X_{\reg}$. Its complement $X_{\sing}:=X\setminus X_{\reg}$ is the singular locus, which we endow with the reduced scheme structure. For a morphism $X\to\Spec k$ over a field $k$, its smooth locus is denoted by $X_{\sm}$.

The following lemma is extracted from the proof of \cite[Thm.~3.4]{GK23}. We state the result in a form suited for our applications.

\begin{lemma}\label{divisor_Ra}
Let $X$ be a locally Noetherian excellent scheme of pure dimension $n$. Assume $X$ is $R_a$ for some $a\ge0$. Let $D$ be an effective Cartier divisor on $X$ that contains no generic point of $X_{\sing}$ of codimension $a+1$ in $X$. If $X_{\reg}\cap D$ is $R_a$, so is $D$.
\end{lemma}

\begin{proof}
Since $X$ is $R_a$, $\dim X_{\sing}\le n-a-1$. We have $\dim(X_{\sing}\cap D)\le n-a-2$: this is clear if $\dim X_{\sing}\le n-a-2$, and thus it can fail only if $X_{\sing}$ has a generic point of dimension $n-a-1$ contained in $D$, which cannot happen by the assumption on $D$. Thus the codimension of $X_{\sing}\cap D$ inside $D$ is at least $(n-1) - (n-a-2) = a+1.$ As $D=(X_{\reg}\cap D)\sqcup (X_{\sing}\cap D)$ and $X_{\sing}\cap D$ has codimension $\ge a+1$ in $D$, any point $x\in D$ of codimension $\le a$ lies in $X_{\reg}\cap D$. Since $X$ is excellent, $X_{\reg}\cap D$ is open in $D$. Thus $\oo_{D, x}=\oo_{X_{\reg}\cap D,x}$, $x\in X_{\reg}\cap D$ also has codimension $\le a$ and, as $X_{\reg}\cap D$ is $R_a$, $\oo_{D, x}$ is regular. 
\end{proof}

\begin{corollary}\label{Bertini_Ra}
Let $X\subseteq\pp^N_k$ be a quasi-projective scheme of pure dimension $n$ over a field $k$. Assume $X$ is $R_a$ for some $a\ge0$. Let $H\subset\pp^N_k$ be a hypersurface that contains no associated point of $X$ and no generic point of $X_{\sing}$ of codimension $a+1$ in $X$. If $X_{\reg}\cap H$ is $R_a$, so is $X\cap H$. 
\end{corollary}

\begin{proof}
Since a hypersurface $H$ that does not contain any associated point of $X$ defines an effective Cartier divisor $D=X\cap H$, Lemma~\ref{divisor_Ra} applies.
\end{proof}

\subsection{Infinite fields} Suppose $k$ is an infinite field and $P(H)$ is a statement depending on $H$. We say that for a \emph{general} hypersurface $H\subset\pp^N_k$ of degree $d$, $P(H)$ holds if there exists a nonempty open subscheme
$U\subset\pp(H^0(\pp^N_k,\oo_{\pp^N_k}(d)))$ such that $P(H)$ holds for every $k$-point $H\in U(k)$.

\begin{lemma}\label{smooth_codim_Bertini_infinite}
Let $X\subseteq\pp^N_k$ be a quasi-projective scheme of pure dimension $n$ over an infinite field $k$. If $X$ is smooth (resp.~regular) in codimension $a$, then for each integer $d\ge1$, a general hypersurface $H\subset\pp^N_k$ of degree $d$ has the property that $X\cap H$ is smooth (resp.~regular) in codimension $a$.
\end{lemma}

\begin{proof}
By replacing $\pp^N_k$ by its $d$-uple embedding, we may assume $d=1$. The smooth locus $X_{\sm}$ of $X\to\Spec k$ is open in $X$, so $Z:=X\setminus X_{\sm}$ is closed, which we endow with the reduced scheme structure. Since $X$ is smooth in codimension $a$, $\dim Z\le n-a-1$. By the Bertini smoothness theorem \cite[Cor.~6.11(2)]{Jou83} applied to $X_{\sm}$, there exists a nonempty open subscheme
$U\subset\pp(H^0(\pp^N_k,\oo_{\pp^N_k}(1)))=:\pp$ such that every $H\in U(k)$ has $X_{\sm}\cap H$ smooth. On the other hand, the set of hyperplanes containing any of the finitely many associated points of $X$ or of $Z$ forms a proper linear subspace $F$ of $\pp$. Thus $U\setminus F$ is nonempty open in $\pp$. Since $k$ is infinite, $U\setminus F$ contains a $k$-rational point; let $H$ be such a point. Since $H$ avoids all generic points of $X$ and $Z$, $X\cap H$ has pure dimension $n-1$ and $\dim Z\cap H\le n-a-2$. In particular, $Z\cap H$ contains no point of $X\cap H$ of codimension $\le a$. Thus every point of $X\cap H$ of codimension $\le a$ lies in $X_{\sm}\cap H$, which is smooth by construction. Thus $X\cap H$ is smooth in codimension $a$. The same argument applies when $X$ is regular in codimension $a$, provided we replace ``smooth'' with ``regular'', the smooth locus $X_{\sm}$ with the regular locus $X_{\reg}$, and Bertini's theorem for smoothness with that for regularity (\cite[Thm.~1]{Sei50}; see also \cite[Lem.~2.4]{GK23}).
\end{proof}

\begin{lemma}\label{separable_residue_field_Bertini_infinite_field} 
Let $X\subseteq\pp^N_k$ be a quasi-projective scheme over an infinite field $k$. Suppose $x\in X$ is a point with $\ka(x)$ separable over $k$. Then for every integer $d\ge1$, a general hypersurface $H\subset\pp^N_k$ of degree $d$ has the property that all generic points $h$ of $\overline{\{x\}}\cap H$ have $\ka(h)$ separable over $k$.
\end{lemma}

\begin{proof}
Equip the closure $\overline{\{x\}}$ with the reduced scheme structure. Its function field $\ka(x)$ is separable over $k$ by assumption, so $\overline{\{x\}}$ is geometrically reduced over $k$ (\cite[\href{https://stacks.math.columbia.edu/tag/04KS}{Tag 04KS}]{stacks-project} item~(4)). Applying \cite[Cor.~6.11(2)]{Jou83} to the closed embedding $\overline{\{x\}}\hookrightarrow\pp^N_k\overset{\text{$d$-uple}}{\hookrightarrow}\pp^M_k$, we find for each $d\ge1$ a dense open subscheme $R_d\subset\pp(H^0(\pp^N_k,\oo_{\pp^N_k}(d)))$ such that every $H\in R_d(k)$ has $\overline{\{x\}}\cap H$ geometrically reduced over $k$. In particular, $\oo_{\overline{\{x\}}\cap H, h}$ is geometrically reduced over $k$ for all $h\in\overline{\{x\}}\cap H$ \cite[\href{https://stacks.math.columbia.edu/tag/035W}{Tag 035W}]{stacks-project}. If $h\in\overline{\{x\}}\cap H$ is a generic point, $\oo_{\overline{\{x\}}\cap H, h}$ is a field and thus is identified with $\ka(h)$. Thus $\ka(h)$ is separable over $k$ by \cite[\href{https://stacks.math.columbia.edu/tag/030W}{Tag 030W}]{stacks-project}.
\end{proof}

The following is the main technical result of this note. For its analogue over finite fields, see Theorem~\ref{N1_Bertini_finite_field}.

\begin{theorem}\label{N1_Bertini_infinite_field}
Let $k$ be an infinite field, and let $X\subseteq\pp^N_k$ be a quasi-projective scheme of pure dimension $\ge2$ over $k$. If $X$ is geometrically~$N_1$ (resp.~$N_1$), then for every integer $d\ge1$ and a general hypersurface $H\subset\pp^N_k$ of degree $d$, $X\cap H$ is geometrically $N_1$ (resp.~$N_1$). 
\end{theorem}

\begin{proof}
Each node $x$ of $X$ is a generic point of the singular locus $X_{\sing}$, so the set of nodes of $X$ is finite and the scheme $(X\setminus\bigcup_{x \text{ a node}}\overline{\{x\}})$ is geometrically $R_1$ (resp.~$R_1$). Recall that geometrically $R_1$ is equivalent to smooth in codimension one \cite[\href{https://stacks.math.columbia.edu/tag/038X}{Tag 038X}]{stacks-project}. Therefore Lemma~\ref{smooth_codim_Bertini_infinite} applies to $(X\setminus\bigcup_{x \text{ a node}}\overline{\{x\}})$, giving for each $d\ge1$ a dense open subscheme $U_d\subset\pp(H^0(\pp^N_k,\oo_{\pp^N_k}(d)))$ such that every $H\in U_d(k)$ has the property that $(X\setminus\bigcup_{x \text{ a node}}\overline{\{x\}})\cap H$ is geometrically $R_1$ (resp.~$R_1$).

Each codimension-one point of $X\cap H$ is either a codimension-one point of $(X\setminus\bigcup_{x \text{ a node}}\overline{\{x\}})\cap H$, or a codimension-zero point of $\overline{\{x\}}\cap H$ for some node $x\in X$. Therefore, it remains to show that for 
each node $x\in X$, a general hypersurface $H$ has the property that all generic points $h$ of $\overline{\{x\}}\cap H$ are nodes of $X\cap H$.

By Theorem~\ref{Geometric_characterization_of_nodes}, $x\in X$ is a node if and only if there exist pointed \'etale $k$-morphisms $(X,x)\overset{\phi}{\longleftarrow}(V,v)\overset{\psi}{\longrightarrow}(Y,y)$, where
\[
Y=\Spec k[a,b,t_1,\dots,t_n]/(ab),\quad y=(a,b)\in Y.
\]
For each $1\le i\le n$, put 
\[
Y_i=\Spec k[a,b,t_1,\dots,\widehat{t_i},\dots,t_n]/(ab),\quad y_i=(a,b)\in  Y_i
\]
so that $Y\cong Y_i\times_k\A^1_{t_i}$. Denote by $\operatorname{pr}_i\colon Y\to Y_i$ the projection. For a hypersurface $H\subset\pp^N_k$, set $H_X:=X\cap H$ and $H_V:=H_X\times_XV$. We want to find hypersurfaces $H$ such that each generic point $h$ of $\overline{\{x\}}\cap H$ admits the following diagram (where all arrows are pointed morphisms except for the hook arrows, which denote non-pointed closed embeddings): 
\[
\begin{tikzcd}
      & (V,v)\arrow[ld,"\phi"swap]\arrow[rd,"\psi"] & \\
      (X,x)& (H_V,h_v)\arrow[u,hook]\arrow[ld] & (Y,y)\arrow[d,"\operatorname{pr}_i"]\\
     (H_X,h)\arrow[u,hook] & &   (Y_i,y_i)
\end{tikzcd}
\] 
Here, for $H_X$ to be an effective Cartier divisor of $X$, $H$ needs to avoid all associated points of $X$. For $\dim \overline{\{x\}}\cap H=\dim \overline{\{x\}}-1$, $H$ needs to avoid all nodes of $X$. Moreover, to ensure the existence of a (necessarily codimension-one) point $h_v\in H_V$ mapping to $h\in H_X$, we need $h\in\phi(V)$, which can be achieved as follows. Since $\phi$ is \'etale, $\phi(V)\subseteq X$ is open, so $\overline{\{x\}}\setminus\phi(V)\subset\overline{\{x\}}$ is proper closed with all irreducible components having dimension $\le\dim\overline{\{x\}}-1$. For each $d\ge1$, the set of hypersurfaces of degree $d$ containing an associated point of $X$, a node $x$, or a generic point of $\overline{\{x\}}\setminus\phi(V)$ of dimension $\dim\overline{\{x\}}-1$ is a finite union of proper linear subspaces $L_{d,j}\subset\pp(H^0(\pp^N_k,\oo_{\pp^N_k}(d)))$. By choosing $H$ outside the finite union $\bigcup_j L_{d,j}$, we ensure that $H$ contains none of those points. Thus $\overline{\{x\}}\cap H$ has pure dimension $\dim \overline{\{x\}}-1$ and, if any such $h$ lies in $\overline{\{x\}}\setminus\phi(V)$, it must be a generic point of $\overline{\{x\}}\setminus\phi(V)$ of dimension $\dim\overline{\{x\}}-1$. Therefore, our choice of $H$ ensures that all such $h$ lie in $\phi(V)$. Let $h_v\in V$ be a point such that $\phi(h_v)=h$; then $h_v\in H_V$. Consequently, for each $H$ outside $\bigcup_j L_{d,j}$, each node $x\in X$, and each generic point $h$ of $\overline{\{x\}}\cap H$, there is a pointed \'etale $k$-morphism $(H_V,h_v)\to(H_X,h)$ giving the above diagram.

We claim that, for each $(H_X,h)$ with $(H_V,h_v)$ obtained as above, if $g\in\m_{V,h_v}\subset\oo_{V,h_v}$ is a local equation of $H_V$ at $h_v$, then its differential $d_{V/Y_i}(g)$ has nonzero image in $\Omega_{V/Y_i,h_v}\otimes_{\oo_{V,h_v}}\kappa(h_v)$ for some $i$, where $\Omega_{V/Y_i}$ is the sheaf of relative differentials. Granting the claim, we finish the proof by ``slicing $V\to Y_i$'': Applying \cite[\href{https://stacks.math.columbia.edu/tag/057C}{Tag 057C}]{stacks-project} to the smooth morphism of relative dimension 1 $\operatorname{pr}_i\circ\,\psi\colon V\to Y_i$ at $h_v$, we find an affine open neighborhood $V'\subset V$ of $h_v$ such that $H_{V'}:=H_V\cap V'\to Y_i$ is smooth of relative dimension 0. This yields pointed \'etale $k$-morphisms
\[
(H_X,h)\leftarrow(H_{V'},h_v)\to(Y_i,w)
\]
where $w:=(\operatorname{pr}_i\circ\,\psi)(h_v)$. We verify that $w=y_i=(a,b)\in Y_i$ is the node. Indeed, as \'etale morphisms preserve the dimensions of local rings, $w\in Y_i$ is a codimension-one point. Suppose $w$ is a smooth point. Then $h\in H_X$ is a smooth point as smoothness is \'etale-local; in particular $\oo_{H_X,h}$ is regular. Since $H$ avoids all associated points of $X$, its local equation $f\in\oo_{X,h}$ must be a nonzerodivisor, so $\oo_{H_X,h}\cong\oo_{X,h}/(f)$ implies $\oo_{X,h}$ is regular \cite[\href{https://stacks.math.columbia.edu/tag/00NU}{Tag 00NU}]{stacks-project}, which contradicts that $h\in\overline{\{x\}}\subset X_{\sing}$. Thus $w\in Y_i$ must be the singular codimension-one point $y_i=(a,b)$, making $h\in H_X$ a node by Theorem~\ref{Geometric_characterization_of_nodes}.

Before proving the claim, we remark that we may and will assume that $h_v\in\overline{\{v\}}$. In fact, since the \'etale morphism $\phi\colon V\to X$ maps $h_v$ to $h$ and $h\in\overline{\{x\}}$, by the going-down property \cite[\href{https://stacks.math.columbia.edu/tag/00HS}{Tag 00HS}]{stacks-project} there exists a point $\tilde{v}\in V$ mapping to $x$ such that $h_v\in\overline{\{\tilde{v}\}}$. Moreover, since $X\overset{\phi}{\leftarrow} V\overset{\psi}{\rightarrow} Y$ are \'etale and $x\in X_{\sing}$, $\psi(\tilde{v})$ must be the singular codimension-one point $y=(a,b)\in Y$. Thus replacing $v$ with $\tilde{v}$ justifies this assumption.

It remains to prove the claim. Intuitively, if all partial derivatives of $g$ at $h_v$ vanish, then $H_V$ must be tangent to $\overline{\{v\}}$ at $h_v$; the rest of the proof makes this precise by establishing the contrapositive. First, we choose $H$ such that it is not tangent to $\overline{\{x\}}$ for any node $x\in X$, as follows. Endowed with its reduced scheme structure, the closure $\overline{\{x\}}$ is $R_0$, and since $x$ is a node, $\ka(x)/k$ is separable. Thus Lemma~\ref{smooth_codim_Bertini_infinite} and Lemma~\ref{separable_residue_field_Bertini_infinite_field} yield for each $d\ge1$ dense open subschemes $R_{d,x}$ and $S_{d,x}$ of $\pp(H^0(\pp^N_k,\oo_{\pp^N_k}(d)))$ such that for every $H\in (R_{d,x}\cap S_{d,x})(k)$, $\overline{\{x\}}\cap H$ is $R_0$ and each of its generic points $h$ has $\ka(h)/k$ separable.

Next, we relate this non-tangency via $\phi\colon V\to X$ to the non-vanishing of the differential of $g$. Let $H$ be a $k$-point of $\bigcap_{x \text{ a node}}(R_{d,x}\cap S_{d,x})$. Since $\overline{\{x\}}\cap H$ is $R_0$, for each of its generic points $h$, $\oo_{\overline{\{x\}}\cap H, h}$ is a field. Let $f\in \oo_{X,h}$ be a local equation of $H_X$ at $h$ and $\ol f\in \oo_{\overline{\{x\}}, h}$ its image. Since $\oo_{\overline{\{x\}}, h}/(\ol f)=\oo_{\overline{\{x\}}\cap H, h}$ is a field, $\ol f$ generates the maximal ideal $\m_{\overline{\{x\}},h}$. Let $g:=\phi^*f\in\oo_
{V,h_v}$ which is a local equation of $H_V$ at $h_v$. Let $\ol g\in \oo_{\overline{\{v\}}, h_v}$ be its image, where $\overline{\{v\}}$ is endowed with its reduced scheme structure. Since the natural morphisms $\overline{\{v\}}\hookrightarrow V\times_X\overline{\{x\}}$ and $V\times_X\overline{\{x\}}\to \overline{\{x\}}$ are unramified, so is their composition $\phi|_{\overline{\{v\}}}\colon\overline{\{v\}}\to\overline{\{x\}}$. Hence $\m_{\overline{\{v\}},h_v}=\m_{\overline{\{x\}},h}\oo_{\overline{\{v\}},h_v}$, which is generated by $\ol g=\phi^*\ol f$. In particular, its class $[\ol g]\in \m_{\overline{\{v\}},h_v}/\m^2_{\overline{\{v\}},h_v}$ is nonzero. By construction $\ka(h_v)/\ka(h)$ and $\ka(h)/k$ are separable, so $\ka(h_v)/k$ is separable, and hence \cite[\href{https://stacks.math.columbia.edu/tag/00TU}{Tag 00TU}]{stacks-project} applies to $\oo_{\overline{\{v\}},h_v}$ and yields that the natural differential map $\m_{\overline{\{v\}},h_v}/\m^2_{\overline{\{v\}},h_v}\to\Omega_{\overline{\{v\}}/k,h_v}\otimes_{\oo_{\overline{\{v\}},h_v}}\ka(h_v)$ is injective. Thus, $[\ol g]$ has nonzero image in $\Omega_{\overline{\{v\}}/k,h_v}\otimes_{\oo_{\overline{\{v\}},h_v}}\ka(h_v)$.

\smallskip
Finally, this non-vanishing yields the claimed nonzero image of $d_{V/Y_i}(g)$. Since $\psi\colon(V,v)\to(Y,y)$ is \'etale, its base change $V\times_Y\overline{\{y\}}\to\overline{\{y\}}$ is \'etale. Because $\overline{\{y\}}\cong\A^n_k$ is regular, $V\times_Y\overline{\{y\}}$ is regular \cite[\href{https://stacks.math.columbia.edu/tag/025N}{Tag 025N}]{stacks-project}. In particular, the irreducible components of $V\times_Y\overline{\{y\}}$ are disjoint and open; since $\overline{\{v\}}$ is one such component, the restriction $\psi|_{\overline{\{v\}}}\colon\overline{\{v\}}\to\overline{\{y\}}$ is \'etale. As $\psi$ and $\psi|_{\overline{\{v\}}}$ are \'etale, we find $\psi^*\Omega_{Y/Y_i}\cong\Omega_{V/Y_i}$ and $\psi^*\Omega_{\overline{\{y\}}/k}\cong\Omega_{\overline{\{v\}}/k}$, respectively. Now let $B=k[a,b,t_1,\dots,t_n]/(ab)$ and $B_i=k[a,b,t_1,\dots,\widehat{t_i},\dots,t_n]/(ab)$. Setting $I=(a,b)B$ and $C=B/I$, there is a canonical exact sequence of $C$-modules
\[
I/I^2\overset{\delta}{\to}\Omega_{B/k}\otimes_BC\to\Omega_{C/k}\to0
\]
where the image of $\delta$ is $C(da\otimes1)\oplus C(db\otimes1)$. On the other hand, the natural map $\Omega_{B/k}\twoheadrightarrow\Omega_{B/B_i}$ sends $dt_i$ to $dt_i$ and sends $da$, $db$, and $dt_j$ for $j\ne i$ to $0$. The induced map $\Omega_{B/k}\twoheadrightarrow\bigoplus_{i=1}^n\Omega_{B/B_i}$ tensored with $C$ is a surjection with kernel the image of $\delta$, hence induces a canonical isomorphism
\[
\Omega_{C/k}\cong\bigoplus_{i=1}^n\left(\Omega_{B/B_i}\otimes_{B}C\right).
\]  
This induces a canonical isomorphism
\[
\Omega_{\overline{\{y\}}/k}\cong\bigoplus_{i=1}^n\left(\Omega_{Y/Y_i}\otimes_{\oo_Y}\oo_{\overline{\{y\}}}\right).
\]  
Applying $\psi^*$ and using the canonical isomorphisms $\psi^*\Omega_{\overline{\{y\}}/k}\cong\Omega_{\overline{\{v\}}/k}$ and $\psi^*\Omega_{Y/Y_i}\cong\Omega_{V/Y_i}$, we obtain 
\[
\Omega_{\overline{\{v\}}/k}\cong\bigoplus_{i=1}^n\left(\Omega_{V/Y_i}\otimes_{\oo_V}\oo_{\overline{\{v\}}}\right).
\] 
Localizing at $h_v$ and tensoring with $\kappa(h_v)$ over $\oo_{\overline{\{v\}},h_v}$ gives  
\begin{equation}\label{Omega}
\Omega_{\overline{\{v\}}/k,h_v}\otimes_{\oo_{\overline{\{v\}},h_v}}\kappa(h_v)\cong\bigoplus_{i=1}^n\left(\Omega_{V/Y_i,h_v}\otimes_{\oo_{V,h_v}}\ka(h_v)\right).
\end{equation} 
Since this canonical isomorphism is compatible with the differential maps, the image of $g\in\oo_{V,h_v}$ via 
\[
\oo_{V,h_v}\to\oo_{\overline{\{v\}},h_v}\to\Omega_{\overline{\{v\}}/k,h_v}\to\Omega_{\overline{\{v\}}/k,h_v}\otimes_{\oo_{\overline{\{v\}},h_v}}\kappa(h_v)
\]
on the left-hand side of (\ref{Omega}) being nonzero implies that its image via
\[
\oo_{V,h_v}\to\Omega_{V/Y_i,h_v}\to\Omega_{V/Y_i,h_v}\otimes_{\oo_{V,h_v}}\kappa(h_v)
\]
must be nonzero in at least one summand of the right-hand side of (\ref{Omega}). This proves the claim.

In sum, for each $d\ge1$, the subscheme of $\pp(H^0(\pp^N_k,\oo_{\pp^N_k}(d)))$ 
\[
\left(U_d\cap\bigcap_{x \text{ a node}}(R_{d,x}\cap S_{d,x})\right)\setminus\bigcup_jL_{d,j}
\]
is a dense open subscheme, each of whose $k$-points $H$ has the property that $X\cap H$ is geometrically $N_1$ (resp.~$N_1$). This completes the proof.
\end{proof}

\subsection{Finite fields}

The following lemma is a key technical ingredient for the proof of Theorem~\ref{N1_Bertini_finite_field}. The underlying finite-field Bertini techniques were originally developed in \cite{Poo04} and \cite{Wut14}; our proof adapts the framework of \cite{GK23}.

\begin{lemma}\label{finite_field_Bertini_R1_except_codimension-one}
Let $k$ be a finite field, and let $W$ be a finite set of points of $\pp^N_k$. Let $X\subseteq\pp^N_k$ be a quasi-projective scheme of pure dimension $\ge2$ over $k$. Assume $X$ is $R_1$ except possibly at a finite set of codimension-one points $x_1,\dots,x_s\in X$. Then for all sufficiently large integers $d$, there exists $f\in H^0(\pp^N_k,\oo_{\pp^N_k}(d))$ defining a hypersurface $H_f$ that satisfies

\begin{enumerate}
    \item $W\cap H_f=\varnothing$.
    \item $X_{\sm}\cap H_f$ is smooth. 
    \item $\overline{\{x_i\}}_{\sm}\cap H_f$ is $R_0$ for all $i=1,\dots,s$, where $\overline{\{x_i\}}_{\sm}$ denotes the smooth locus of the closure $\overline{\{x_i\}}$ (with the reduced scheme structure) of $x_i$.
\end{enumerate}
\end{lemma}

\begin{proof}
This follows the strategy of the proof of \cite[Thm.~4.5]{GK23}, using the technical inputs from \cite[Prop.~4.8]{GK23} to our specific setup. For each $i=1,\dots,s$, put $F_i:=\overline{\{x_i\}}\cap(\bigcup_{j\ne i}\overline{\{x_j\}})$. Denote by $\{\eta_j\}_j$ the finite set of generic points of all $F_1,\dots,F_s$. Express the finite set $W\cup\{\eta_j\}_j=W_1\sqcup W_2$, where $W_1$ (resp.~$W_2$) consists of the closed (resp.~non-closed) points of $\pp^N_k$ lying in $W\cup\{\eta_j\}_j$. Write $W_1=\{w_1,\dots,w_r\}$ and define $T:=\prod_{i=1}^r\ka(w_i)^{\times}\subset H^0(W_1,\oo_{W_1})$. Set
\[
U_0=X_{\sm}\setminus W_1,\quad U_i=\overline{\{x_i\}}_{\sm}\setminus(F_i\cup W_1),\,\, i=1,\dots,s.
\]
Since $x_i\in X\setminus X_{\sm}$, $X_{\sm}$ is disjoint from $\overline{\{x_i\}}$. Thus $U_0,U_1,\dots,U_s$ are smooth, pure-dimensional, pairwise disjoint subschemes of $\pp^N_k$ with $W_1\cap U_i=\varnothing$ for all $i$. Denote by $S_{\operatorname{homog}}$ the graded ring of homogeneous polynomials of $\pp^N_k$. Applying \cite[Prop. 4.8]{GK23} with their $Y=W_1$ and $Z=\varnothing$ yields that the set $\mathcal{P}_0$ of all $f\in S_{\operatorname{homog}}$ such that $f|_{W_1}\in T$ and $U_i\cap H_f$ is smooth for all $i=0,1,\dots, s$ has positive density $\mu(\mathcal{P}_0)>0$. 

For any point $x\in\pp^N_k$, we put $\mathcal{P}_x=\{f\in S_{\operatorname{homog}}\mid x\not\in H_f\}$; then define $\mathcal{P}=(\bigcap_{x\in W_2}\mathcal{P}_x)\cap\mathcal{P}_0$. Since each $x\in W_2$ is non-closed, $\mu(\mathcal{P}_x)=1$ by \cite[Lem.~4.2]{GK23}. It then follows from \cite[Lem.~4.1(2)]{GK23} that $\mu(\mathcal{P})=\mu(\mathcal{P}_0)>0$. This positive density implies (cf. \cite[Lem.~4.1(1)]{GK23}) that for all sufficiently large integers $d$ there exists $f\in H^0(\pp^N_k,\oo_{\pp^N_k}(d))$ which defines a hypersurface $H_f$ satisfying $(a)$ $f|_{W_1}\in T$, $(b)$ $x\not\in H_f$ for all $x\in W_2$, and $(c)$ $U_i\cap H_f$ is smooth for all $i=0,1,\dots, s$. It remains to show that $H_f$ satisfies all items~$(1), (2),$ and $(3)$ of the lemma.

Item $(a)$ implies $W_1\cap H_f=\varnothing$, and item~$(b)$ implies $W_2\cap H_f=\varnothing$. Thus $(W\cup \{\eta_j\}_j)\cap H_f=\varnothing$; in particular $H_f$ satisfies item~$(1)$. By item~$(c)$, $U_0\cap H_f=(X_{\sm}\cap H_f)\setminus (W_1\cap H_f)$ is smooth. Since $W_1\cap H_f=\varnothing$, $X_{\sm}\cap H_f$ is smooth. Thus $H_f$ verifies item~$(2)$. Likewise, by item~$(c)$, $\overline{\{x_i\}}_{\sm}\cap H_f$ is smooth everywhere except possibly along $(F_i\cup W_1)\cap H_f=F_i\cap H_f$. As $H_f$ avoids all generic points $\eta_j$ of $F_1,\dots,F_s$, $\dim(F_i\cap H_f)<\dim(\overline{\{x_i\}}_{\sm}\cap H_f)$. Thus $\overline{\{x_i\}}_{\sm}\cap H_f$ is $R_0$ and $H_f$ satisfies item~$(3)$. 
\end{proof}

With this technical lemma, we obtain the following finite-field analogue of Theorem~\ref{N1_Bertini_infinite_field}. We omit the word ``geometrically'' over perfect fields.

\begin{theorem}\label{N1_Bertini_finite_field}
Let $k$ be a finite field, and let $X\subseteq\pp^N_k$ be a quasi-projective scheme of pure dimension $\ge2$ over $k$. If $X$ is $N_1$ (resp.~$R_1$), then for all sufficiently large integers $d$, there exists $f\in H^0(\pp^N_k,\oo_{\pp^N_k}(d))$ defining a hypersurface $H_f$ such that $X\cap H_f$ is $N_1$ (resp.~$R_1$).
\end{theorem}

\begin{proof}
The proof of the infinite-field case (Theorem~\ref{N1_Bertini_infinite_field}) shows that for any $k$-point $H$ of the dense open subscheme \mbox{$\left(U_d\cap\bigcap_{x \text{ a node}}(R_{d,x}\cap S_{d,x})\right)\setminus\bigcup_jL_{d,j}$}, $X\cap H$ is geometrically $N_1$. Since this implication is valid over any field, the proof for the finite-field case reduces to proving the existence of a $k$-point $H_f$ satisfying the corresponding conditions. Specifically, let $x_1,\dots,x_s$ be the nodes of $X$. Equip the closures $\overline{\{x_i\}}$ with the reduced scheme structures; in particular, each $\overline{\{x_i\}}$ is $R_0$. We require $H_f$ to satisfy the following four properties: 

\begin{enumerate}
    \item[$(A)$] $(X\setminus\bigcup_{i=1}^s\overline{\{x_i\}})\cap H_f$ is $R_1$. 
    \item[$(B)$] $\overline{\{x_i\}}\cap H_f$ is $R_0$ for all $i=1,\dots,s$.
    \item[$(C)$] Every generic point $h$ of $\overline{\{x_i\}}\cap H_f$ has residue field $\ka(h)$ separable over $k$ for all $i=1,\dots,s$.
    \item[$(D)$] $H_f$ contains no associated point of $X$, no node $x_i$, and no generic point of $\overline{\{x_i\}}\setminus\phi_i(V_i)$ of dimension $\dim\overline{\{x_i\}}-1$ for all $i=1,\dots,s$, where $\phi_i\colon (V_i,v_i)\to(X,x_i)$ is the \'etale morphism coming from Theorem~\ref{Geometric_characterization_of_nodes}.
\end{enumerate}

To construct such an $H_f$, we define a finite set $W\subset\pp^N_k$ of ``bad'' points that $H_f$ needs to avoid. Let $W$ be the union of the following:

\begin{enumerate}
\item[$(\alpha)$] Associated points of $X$, nodes $x_i$, and generic points of $\overline{\{x_i\}}\setminus\phi_i(V_i)$ of dimension $\dim\overline{\{x_i\}}-1$ for all $i=1,\dots,s$.
\item[$(\beta)$] Associated points of $(X\setminus\bigcup_{i=1}^s\overline{\{x_i\}})$ and generic points of $(X\setminus\bigcup_{i=1}^s\overline{\{x_i\}})_{\sing}$ of codimension 2 in $(X\setminus\bigcup_{i=1}^s\overline{\{x_i\}})$.
\item[$(\gamma)$] Associated points of $\overline{\{x_i\}}$ and generic points of $\overline{\{x_i\}}_{\sing}$ of codimension 1 in $\overline{\{x_i\}}$, for all $i=1,\dots,s$. 
\end{enumerate}

Applying Lemma~\ref{finite_field_Bertini_R1_except_codimension-one} to $W$ shows that for all sufficiently large integers $d$, there exists $f\in H^0(\pp^N_k,\oo_{\pp^N_k}(d))$ defining a hypersurface $H_f$ that satisfies

\begin{enumerate}
    \item[$(1)$] $W\cap H_f=\varnothing$.
    \item[$(2)$] $X_{\reg}\cap H_f$ is regular. 
    \item[$(3)$] $\overline{\{x_i\}}_{\reg}\cap H_f$ is $R_0$ for all $i=1,\dots,s$.
\end{enumerate} 
Here, the regular and smooth loci coincide since $k$ is perfect. It remains to show that such an $H_f$ satisfies all properties $(A)$ through $(D)$; the argument from the proof of Theorem~\ref{N1_Bertini_infinite_field} then applies to show that $X\cap H_f$ is $N_1$.

First, since $k$ is perfect, every field extension of $k$ is separable over $k$. Thus $H_f$ satisfies property $(C)$.

Next, item~$(1)$ implies that $H_f$ avoids all points in items~$(\alpha)$ through $(\gamma)$. In particular, item~$(\alpha)$ implies that $H_f$ satisfies property $(D)$.

Finally, we apply Corollary~\ref{Bertini_Ra} to verify properties $(A)$ and $(B)$. Since the scheme $(X\setminus\bigcup_{i=1}^s\overline{\{x_i\}})$ is $R_1$, $H_f$ avoids the points in item~$(\beta)$, and $(X\setminus\bigcup_{i=1}^s\overline{\{x_i\}})_{\reg}\cap H_f=X_{\reg}\cap H_f$ is $R_1$ by item~$(2)$, we find that $(X\setminus\bigcup_{i=1}^s\overline{\{x_i\}})\cap H_f$ is $R_1$. In particular, if $X$ is $R_1$, then $X\cap H_f$ is $R_1$. Likewise, since $\overline{\{x_i\}}$ is $R_0$, items~$(\gamma)$ and $(3)$ imply that $\overline{\{x_i\}}\cap H_f$ is $R_0$. Thus $H_f$ satisfies properties $(A)$ and $(B)$. This completes the proof.
\end{proof}

\section{Geometric \texorpdfstring{$N_1$}{N1} reduction}

This section defines the notion of geometric $N_1$ reduction (Definition~\ref{def_geo_N1_reduction}) and proves its Bertini theorem, Theorem~\ref{geo_N1_reduction_Bertini_DVR}.

Recall that $S$ denotes a Dedekind scheme with fraction field $K$, and let $X_K$ be a projective, geometrically normal, and geometrically connected scheme over $K$. In particular, $X_K$ is geometrically integral.

\begin{definition}\label{def_model}
A \emph{model} of $X_K$ over $S$ is a flat projective scheme $X$ over $S$ whose generic fibre is isomorphic to $X_K$ over $K$.
\end{definition}

\begin{remark}
For any model $X$ of $X_K$ over $S$ (Definition~\ref{def_model}), the fibre $X_s$ over any closed point $s\in S$ is of pure dimension $\dim X_K$ \cite[Cor.~4.3.14]{Liu02} and geometrically connected \cite[\href{https://stacks.math.columbia.edu/tag/0AYI}{Tag 0AYI}]{stacks-project}.
\end{remark}

\begin{definition}\label{def_geo_N1_reduction}
We say that $X_K$ has \emph{geometrically $N_1$ (resp.~geometrically $R_1$, $N_1$, $R_1$) reduction} at a closed point $s\in S$ if there exists a model $X$ of $X_K$ over the DVR $\oo_{S,s}$ whose special fibre $X_s$ is geometrically $N_1$ (resp.~geometrically $R_1$, $N_1$, $R_1$).
\end{definition}

The properties of our generic fibre $X_K$ are preserved under taking general hypersurface sections provided $\dim X_K\ge2$ (see Corollary~\ref{geo_conn_and_normal}).

\begin{lemma}\label{geo_normal_Bertini}
Let $K$ be an infinite field. Let $X_K\subseteq\pp^N_K$ be a projective geometrically normal scheme over $K$. Then for every integer $d\ge1$, a general hypersurface $H_K\subset\pp^N_K$ of degree $d$ has the property that $X_K\cap H_K$ is geometrically normal.
\end{lemma}

\begin{proof}
By replacing $\pp^N_K$ with its $d$-uple embedding, we may assume $d=1$. Let $X_K\subseteq\pp^N_K=\operatorname{Proj}K[x_0, \dots, x_N]$. Then $\pp(H^0(\pp^N_K,\oo_{\pp^N_K}(1)))$ is the dual projective space $(\pp^N_K)^{\vee}=\operatorname{Proj}K[a_0, \dots, a_N]$. The universal hyperplane $\mathcal{H}\subset(\pp^N_K)^{\vee}\times_K\pp^N_K$ is defined by the ideal generated by $\sum_{i=0}^Na_ix_i$. Consider the incidence scheme $I$ as the closed subscheme of $(\pp^N_K)^{\vee}\times_K\pp^N_K$ 
\[
I:=\mathcal{H}\cap((\pp^N_K)^{\vee}\times_KX_K).
\]

Let $\pi\colon I\to(\pp^N_K)^{\vee}$ denote the first projection. Since $I\subset(\pp^N_K)^{\vee}\times_K X_K$ is closed and $X_K$ is proper over $K$, the morphism $\pi$ is proper. Any $K$-point $h$ of $(\pp^N_K)^{\vee}$ corresponds to a hyperplane $H_h\subset\pp^N_K$, and the fibre $I_h:=\pi^{-1}(h)$ is canonically identified with $X_K\cap H_h$. By generic flatness (\cite[\href{https://stacks.math.columbia.edu/tag/052A}{Tag 052A}]{stacks-project}), there exists a dense open subscheme $U\subset(\pp^N_K)^{\vee}$ such that the base change $\pi_U\colon I_U=I\times_{(\pp^N_K)^{\vee}}U\to U$ is flat and of finite presentation. Thus by \cite[Thm.~12.2.4(iv)]{EGAIV3}, the set 
\[
N=\{u\in U\,|\,\text{$I_u$ is geometrically normal over }\kappa(u)\}
\]
is open in $U$, and hence an open subscheme of $(\pp^N_K)^{\vee}$. Every $K$-point $h$ of $N$ corresponds to a hyperplane $H_h$ such that $X_K\cap H_h$ is geometrically normal. Thus, it remains to show that $N$ is nonempty.

Let $L:=K^{\operatorname{perf}}$ denote the perfect closure of $K$. Since $X_K$ is geometrically normal, its base change $X_L$ is normal. By \cite[Cor.~3.4.14]{FOV99}, there exists a dense open subscheme $V\subset(\pp^N_L)^{\vee}$ such that every $H'\in V(L)$ has $X_L\cap H'$ a normal scheme. Since $L$ is a perfect field, $X_L\cap H'$ is geometrically normal over $L$ \cite[\href{https://stacks.math.columbia.edu/tag/038O}{Tag 038O}]{stacks-project}.

Let $U_L$ denote the base change of $U$ and let $g\colon U_L\to U$ be the projection. Since $L$ is infinite, $(V\cap U_L)(L)$ is nonempty. Choose any $\alpha\in (V\cap U_L)(L)$, let $\beta=g(\alpha)\in U(L)$, and let $a,b$ be their image points in $U_L, U$ respectively. Then, denoting $I_L$ the base change of $I$, the fibre $(I_L)_a\cong X_L\cap H'_a$ is geometrically normal over $L$. Moreover, $(I_L)_a\cong I_b\times_{\kappa(b)}L$. By \cite[\href{https://stacks.math.columbia.edu/tag/038P}{Tag 038P}]{stacks-project}, $(I_L)_a$ is geometrically normal over $L$ if and only if $I_b$ is geometrically normal over $\kappa(b)$. Thus $\beta\in N(L)$. Hence $N$ is nonempty. This completes the proof.
\end{proof}

\begin{corollary}\label{geo_conn_and_normal} 
Assume $K$ is an infinite field and $\dim X_K\ge2$. If $X_K\subseteq\pp^N_K$ is a projective, geometrically normal, and geometrically connected scheme over $K$, then for every integer $d\ge1$ and a general hypersurface $H_K\subset\pp^N_K$ of degree $d$, $X_K\cap H_K$ has the same properties.
\end{corollary}

\begin{proof}
By replacing $\pp^N_K$ with its $d$-uple embedding, we may assume $d=1$. Since a normal connected scheme is irreducible \cite[\href{https://stacks.math.columbia.edu/tag/033M}{Tag 033M}]{stacks-project}, $X_K$ is geometrically irreducible over $K$. Since $\dim X_K\ge2$, \cite[Cor.~6.11(3)]{Jou83} yields a dense open subscheme $W\subset (\pp^N_K)^{\vee}$ such that every $H_K\in W(K)$ has $X_K\cap H_K$ geometrically irreducible over $K$. Let $N\subset(\pp^N_K)^{\vee}$ be the nonempty open subscheme as in Lemma~\ref{geo_normal_Bertini}. The intersection $W\cap N$ is a nonempty open subscheme of $(\pp^N_K)^{\vee}$ whose $K$-points have the asserted properties. 
\end{proof}

\begin{remark}[Specialization map]\label{Rem_specialization_map}
Let $R$ be a DVR with fraction field $K$ and residue field $k$, and let $N\ge1$ be an integer. The specialization map 
\[
\operatorname{sp}\colon\pp^N_K(K)\to\pp^N_k(k)
\]
sends a $K$-point $x$ to the restriction of its closure $\overline{\{x\}}$ in $\pp^N_R$ to the special fibre $\pp^N_k$. It is well-defined and surjective. For more details, see for instance \cite[\S~6.1]{GK23}.
\end{remark}

\begin{lemma}[{\cite[Lem.~6.2]{GK23}}]\label{infinite_intersection}
For any point $x\in\pp^N_k(k)$ and any nonempty open subscheme $U\subseteq\pp^N_K$, the intersection $\operatorname{sp}^{-1}(x)\cap U(K)$ is infinite.
\end{lemma}

The main result of this section is the following. 

\begin{theorem}\label{geo_N1_reduction_Bertini_DVR}
Let $S$ be a Dedekind scheme with fraction field $K$ and a closed point $s\in S$. Let $X_K$ be a projective, geometrically normal, and geometrically connected scheme over $K$. Assume $\dim X_K\ge2$. If $X_K$ has geometrically $N_1$ (resp.~geometrically $R_1$, $N_1$, $R_1$) reduction at $s\in S$, then there exists a closed embedding $X_K\subseteq\pp^N_K$ such that for all sufficiently large integers $d$, there exist infinitely many hypersurfaces $H_K\subset\pp^N_K$ of degree $d$ such that $X_K\cap H_K$ has the same reduction type at $s\in S$.
\end{theorem}

\begin{proof}
Let $R$ denote the DVR $\oo_{S,s}$ with residue field $k$ and fraction field $K$. By assumption there exists a flat projective morphism $X\to\Spec R$ with generic fibre $X_K$ and special fibre $X_k$ which is geometrically $N_1$ (resp.~geometrically $R_1$, $N_1$, $R_1$). Fix a closed embedding $X\subseteq\pp^N_R$; this induces embeddings $X_K\subseteq\pp^N_K$ and $X_k\subseteq\pp^N_k$. For any integer $d\ge1$, define $\pp_d:=\pp(H^0(\pp^N_R, \oo_{\pp^N_R}(d)))$. The specialization map $(\pp_d)_K(K)\to(\pp_d)_k(k)$ is well-defined and surjective (cf. Remark~\ref{Rem_specialization_map}).

By Corollary~\ref{geo_conn_and_normal}, for each integer $d\ge1$ there exists a dense open subscheme $U_d\subset(\pp_d)_K$ such that every $H_K\in U_d(K)$ has $X_K\cap H_K$ projective, geometrically normal, and geometrically connected.

By Theorems \ref{N1_Bertini_infinite_field}, \ref{N1_Bertini_finite_field}, and Lemma~\ref{smooth_codim_Bertini_infinite}, there exists $d_0\gg0$ such that for all integers $d\ge d_0$, there exists a hypersurface $H_k\in(\pp_d)_k(k)$ for which $X_k\cap H_k$ is geometrically $N_1$ (resp.~geometrically $R_1$, $N_1$, $R_1$). Note that by its construction, $H_k$ does not contain any associated point of $X_k$.

For a fixed degree $d\ge d_0$, let $H_k\in(\pp_d)_k(k)$ be a hypersurface chosen as above. The set $V_d(K):=\operatorname{sp}^{-1}(H_k)\cap U_d(K)\subset(\pp_d)_K(K)$ is an infinite set by Lemma~\ref{infinite_intersection}. Pick any $H_K\in V_d(K)$, and let $H\subset\pp^N_R$ denote its schematic closure. Then $H$ has generic fibre $H_K$ and, by the definition of the specialization map, special fibre the chosen $H_k$. 

We show that the closed subscheme $X\cap H\subset X$ is flat over $R$. Since $X$ is flat over $R$, $\oo_{X,x}$ is flat over $R$ for any $x\in X$. For points $x\in (X\cap H)_K$, $\oo_{X\cap H, x}$ is a $K$-algebra and $K$ is flat over $R$, so flatness over $R$ follows. Consider $x\in (X\cap H)_k$. Let $\pi$ be a uniformizer of $R$, and let $f\in\oo_{X,x}$ be a local equation of $H$. Since $H_k$ avoids the associated points of $X_k$, the image of $f$ in $\oo_{X,x}/(\pi)=\oo_{X_k,x}$ is a nonzerodivisor. By the criterion for flatness \cite[\href{https://stacks.math.columbia.edu/tag/00MF}{Tag 00MF}]{stacks-project}, $\oo_{X\cap H, x}=\oo_{X,x}/(f)$ is flat over $R$.

Therefore $X\cap H$ is flat projective over $R$ and satisfies:

1. Its generic fibre is isomorphic to $X_K\cap H_K$, which is projective, geometrically normal, and geometrically connected; and

2. Its special fibre is isomorphic to $X_k\cap H_k$, which is geometrically $N_1$ (resp.~geometrically $R_1$, $N_1$, $R_1$). 

Thus $X_K\cap H_K$ has geometrically $N_1$ (resp.~geometrically $R_1$, $N_1$, $R_1$) reduction at $s\in S$.
\end{proof}

\section{Proof of the main result}

We first prove the base case.

\begin{lemma}\label{curve_N1_implies_semi-abelian}
Let $S$ be a Dedekind scheme with fraction field $K$. Let $X_K$ be a 1-dimensional, projective, geometrically normal, and geometrically connected scheme over $K$. If $X_K$ has geometrically $N_1$ (resp.~geometrically $R_1$) reduction at $s\in S$, then its Jacobian $\Pic^0_{X_K/K}=(\Pic^0_{X_K/K})_{\re}$ has semi-abelian (resp.~good) reduction at $s\in S$.    
\end{lemma}

\begin{proof}
Since $\dim X_K=1$, geometric normality implies smoothness; thus $\Pic^0_{X_K/K}$ is an abelian variety \cite[Prop.~9.2/3]{BLR90}.

Let $R$ denote the DVR $\oo_{S,s}$ with residue field $k$. Let $X\to\Spec R$ be a model of $X_K$ over $R$ such that its special fibre $X_k$ is geometrically $N_1$ (resp.~geometrically $R_1$). Then the geometric fibre $X_{\ol k}:=X_k\times_k \ol k$ is $N_1$ by Lemma \ref{geo_N1_implies_base_change_N1}. Thus $X\to\Spec R$ is a semi-stable (resp.~smooth) curve in the sense of \cite[Def.~9.2/6]{BLR90}; that is, all its geometric fibres are reduced, connected, one-dimensional, and have at worst nodes as singularities.

Since $X\to\Spec R$ is flat, $X_K$ is normal, and $X_k$ is reduced, the total space $X$ is normal (cf. \cite[Lem.~4.1.18]{Liu02}). Thus it follows from \cite[Cor.~9.7/2]{BLR90} that $\Pic^0_{X_K/K}$ has semi-abelian reduction at $s\in S$. Moreover, since the identity component $\mathscr{N}^0$ of the N\'eron model $\mathscr{N}$ of $\Pic^0_{X_K/K}$ is isomorphic to $\Pic^0_{X/R}$, if $X\to\Spec R$ is smooth then $\mathscr{N}^0_k\cong\Pic^0_{X_k/k}$ is an abelian variety, i.e. $\Pic^0_{X_K/K}$ has good reduction at $s\in S$.
\end{proof}

For the higher dimensional case, we make two remarks regarding the Albanese variety and the Picard variety.

\begin{remark}\label{existence_of_Albanese}
By a theorem of Serre \cite[Thm.~7]{Ser60}, extended in \cite[Thm.~A.1]{Wit08}, any geometrically reduced and geometrically connected scheme $X_K$ of finite type over a field $K$ admits an Albanese variety $\Alb_{X_K/K}$. See  \cite[Thm.~2.2 and \S~2.1]{ACMV25} for more details.
\end{remark}

\begin{remark}\label{Picard_variety_dual_to_Albanese}
Let $X_K$ be a proper, geometrically normal, and geometrically connected scheme over a field $K$. Then $\Pic^0_{X_K/K}$ is proper over $K$ \cite[Thm.~VI.2.1(ii)]{Gro62}, and its reduced subscheme $(\Pic^0_{X_K/K})_{\re}$, known as the Picard variety, is a group scheme \cite[Prop.~VI.3.1]{Gro62} and hence an abelian variety. Moreover, $(\Pic^0_{X_K/K})_{\re}$ is canonically dual to $\Alb_{X_K/K}$ over $K$ \cite[Thm.~VI.3.3(iii)]{Gro62}. See \cite[Rem.~3.3]{ACMV25} for more details. 
\end{remark}

We now prove the main result of this note.

\begin{proof}[Proof of Theorem~\ref{Main_result}]

We proceed by induction on $\dim X_K$. The base case $\dim X_K=1$ is Lemma~\ref{curve_N1_implies_semi-abelian}.

Suppose $\dim X_K\ge2$. Since $X_K$ has  geometrically $N_1$ (resp.~ geometrically $R_1$) reduction at $s\in S$, by Theorem~\ref{geo_N1_reduction_Bertini_DVR} there exists an integer $d_0\gg0$ such that for every $d\ge d_0$, there exists a hypersurface section $H_K\subset X_K$ of degree $d$ having geometrically $N_1$ (resp.~geometrically $R_1$) reduction at $s\in S$. The inclusion $\iota\colon H_K\hookrightarrow X_K$ induces a morphism of abelian varieties over $K$ 
\[
\iota^*\colon(\Pic^0_{X_K/K})_{\re}\to(\Pic^0_{H_K/K})_{\re}
\]
We claim $\operatorname{ker}\iota^*$ is finite for degree $d\gg0$. It suffices to show that its Zariski tangent space $T_0(\operatorname{ker}\iota^*)$ is trivial. Since $T_0(\operatorname{ker}\iota^*)$ is contained in the kernel of the tangent map $T_0\Pic^0_{X_K/K}\to T_0\Pic^0_{H_K/K}$, it is enough to prove this map has trivial kernel. Since $T_0\Pic^0_{Z/K}\cong H^1(Z,\oo_Z)$ for any proper geometrically reduced scheme $Z$ over a field $K$ \cite[Thm.~8.4/1(a)]{BLR90}, it suffices to show that the natural map $H^1(X_K,\oo_{X_K})\to H^1(H_K,\oo_{H_K})$ is injective. This map is induced by the short exact sequence of sheaves
\[
0\to\oo_{X_K}(-d)\to\oo_{X_K}\to\oo_{H_K}\to0.
\]
Since $X_K$ is normal projective of dimension $\ge2$, the Enriques-Severi-Zariski lemma \cite[Cor.~III.7.8]{Har77} yields $H^1(X_K,\oo_{X_K}(-d))=0$ for all $d\ge d_1\gg0$. This proves the claim.

Now, fix $d\ge\max\{d_0,d_1\}$ and let $H_K$ be the corresponding hypersurface section. By the induction hypothesis, $(\Pic^0_{H_K/K})_{\re}$ has semi-abelian (resp.~good) reduction at $s\in S$. Thus its abelian subvariety $\operatorname{im}\iota^*$ has the same reduction type \cite[Lem.~7.4/2]{BLR90}. Because $\operatorname{ker}\iota^*$ is finite, the surjection $(\Pic^0_{X_K/K})_{\re}\to\operatorname{im}\iota^*$ is an isogeny. Since an isogeny preserves semi-abelian (resp.~good) reduction \cite[Cor.~7.3/7]{BLR90}, $(\Pic^0_{X_K/K})_{\re}$ has semi-abelian (resp.~good) reduction at $s\in S$.

Finally, by Remark~\ref{Picard_variety_dual_to_Albanese}, $(\Pic^0_{X_K/K})_{\re}$ and $\Alb_{X_K/K}$ are dual abelian varieties. Since every abelian variety is isogenous to its dual, the two varieties share the same reduction types. This completes the proof.
\end{proof}

\section{Proofs of corollaries}\label{Proofs_corollaries}

\begin{corollary}\label{complete_intersections}
Let $X_K$ be as in Theorem~\ref{Main_result}. Suppose that $X_K\subseteq\pp^N_K$ has geometrically $N_1$ (resp.~geometrically $R_1$) reduction at $s\in S$. Assume $n=\dim X_K\ge2$. Then there exist infinitely many $(n-1)$-tuples $(H_1, \dots, H_{n-1})$ of hypersurfaces in $\pp^N_K$ such that for every $1\le i\le n-1$, the intersection $Y_i:=X_K\cap H_1\cap\cdots\cap H_i$ satisfies

$(1)$ $Y_i$ has geometrically $N_1$ (resp.~geometrically $R_1$) reduction at $s\in S$;

$(2)$ $\Alb_{Y_i/K}$ and $(\Pic^0_{Y_i/K})_{\re}$ have semi-abelian (resp.~good) reduction at $s\in S$.
\end{corollary}

\begin{proof} 
Since $\dim X_K\ge2$, Theorem~\ref{geo_N1_reduction_Bertini_DVR} provides infinitely many choices for $H_1$ such that $Y_1$ satisfies property $(1)$. Theorem~\ref{Main_result} then applies to $Y_1$ to yield property $(2)$. The assertion follows by repeating this argument.
\end{proof}

\begin{proof}[Proof of Corollary~\ref{locally_stable_implies_geo_N1}] 
By the definition of locally stable reduction, there is a model $X$ of $X_K$ over $\oo_{S,s}$ such that $X_s$ is reduced and $(X,X_s)$ is log canonical. By \cite[Cor.~2.32]{Kol13}, every codimension-one point of $X_s$ is either a regular point or an intrinsic node (Definition~\ref{intrinsic_node}). Since $\ka(s)$ is perfect, regularity coincides with geometric regularity, and an intrinsic node is a node (Lemma~\ref{intrinsic_node_is_node}) as $\ch\ka(s)\ne2$. Thus $X_s$ is geometrically $N_1$, and Theorem~\ref{Main_result} applies.
\end{proof}

\begin{proof}[Proof of Corollary~\ref{semi-stable_implies_semi-abelian}]
The definition of semi-stable reduction implies that each point $x\in X_s$ admits \'etale $\ka(s)$-morphisms 
\[
(X_s,x)\leftarrow (U_s,u)\to(Z,z),
\]
where $Z=\Spec\ka(s)[x_1,x_2,\dots, x_{n+1}]/(x_1x_2\cdots x_r)$ and $z$ is the image of $u$. By Lemma~\ref{snc_implies_geo_N1}, $Z$ is geometrically $N_1$. Since \'etale morphisms preserve the dimensions of local rings, and since smoothness and nodality are \'etale-local properties, $X_s$ is geometrically $N_1$, and Theorem~\ref{Main_result} applies.
\end{proof}

\bibliographystyle{amsalpha}
\bibliography{references}

\end{document}